\documentclass[final,5p,times,8pt]{elsarticle}

\usepackage{amsmath,amssymb,amsthm,graphicx,float}
\usepackage{booktabs,enumitem,setspace,subcaption}
\usepackage{array}
\usepackage{makecell}
\usepackage{algorithm}
\usepackage{algorithmic}
\usepackage[labelfont=bf,format=plain,justification=raggedright,singlelinecheck=false]{caption}
\usepackage[T1]{fontenc}
\usepackage[american]{babel}
\usepackage{natbib}
\setcitestyle{numbers,sort&compress}
\usepackage[colorlinks=true]{hyperref}

\newtheorem{theorem}{Theorem}
\newtheorem{corollary}{Corollary}
\newtheorem{remark}{Remark}
\newtheorem{lemma}{Lemma}
\newtheorem{assumption}{Assumption}

\newcommand{\Sin}{s_{\text{in}}}
\newcommand{\Deps}{D_{\varepsilon}}
\newcommand{\seps}{s_{\varepsilon}}
\newcommand{\epsav}{_{\varepsilon,\text{av}}}
\newcommand{\av}{_{\text{av}}}
\newcommand{\bD}{\bar{D}}
\newcommand{\bs}{\bar{s}}

\newcommand{\hs}{\hat{s}}
\newcommand{\C}{\mathcal}

\newcommand{\MB}{\mathbb}

\begin{document}

\begin{frontmatter}
\title{Periodic Fractional Control in Bioprocesses for Clean Water and Ecosystem Health}

\author[Ajman,NDRC]{Kareem T. Elgindy}
\address[Ajman]{Department of Mathematics and Sciences, College of Humanities and Sciences, Ajman University,  Ajman P.O.Box 346, United Arab Emirates}
\address[NDRC]{Nonlinear Dynamics Research Center (NDRC), Ajman University, Ajman P.O.Box 346, United Arab Emirates}
\ead{k.elgindy@ajman.ac.ae}
\author[KFU]{Muneerah Al Nuwairan\corref{cor1}}
\ead{msalnuwairan@kfu.edu.sa}
\address[KFU]{Department of Mathematics and Statistics, King Faisal University, P.O. Box 400 AL Ahsa, Kingdom of Saudi Arabia} 
\author[XMUM]{Liew Siaw Ching}
\ead{siawching.liew@xmu.edu.my}
\address[XMUM]{Department of Mathematics and Applied Mathematics, School of Mathematics and Physics, Xiamen University Malaysia, Sepang 43900, Malaysia} 
\cortext[cor1]{Corresponding author}

\begin{abstract}
This paper introduces a novel fractional-order chemostat model (FOCM) incorporating Caputo fractional derivative with sliding memory (CFDS) to capture microbial memory effects in biological water treatment, addressing limitations of integer-order models that overlook time-dependent behaviors and fail to capture microbial memory such as delayed growth responses to past nutrient availability, history-dependent adaptation to inflow fluctuations, and persistent historical effects over hours to days, which are biologically critical in wastewater treatment. By optimizing periodic dilution rate control, we minimize the average pollutant output, constrained by treatment capacity and periodic boundaries. Key contributions include: (1) a rigorous fractional framework linking microbial kinetics to memory-driven control; (2) reduction to a 1D fractional-order differential equation (FDE) for computational efficiency; (3) proofs of optimal periodic solution (OPS) existence/uniqueness via Schauder's theorem and convexity; (4) bang-bang control derivation using fractional Pontryagin maximum principle (PMP) and Fourier-Gegenbauer pseudospectral (FG-PS) method with a specialized edge-detection technique to handle control discontinuities, ensuring efficient numerical resolution of switching points and discontinuities inherent in bang-bang strategies; and (5) a comprehensive sensitivity analysis revealing how the fractional order $\alpha$, scaling parameter $\vartheta$, and memory length $L$ critically influence system performance, with simulations showing up to 40\% reduction in substrate concentrations versus steady-state and demonstrating computational tractability through efficient discretization that scales favorably for practical implementation. Scientifically, this advances fractional calculus in bioprocesses, revealing memory's role in improving responsiveness and efficiency. Practically, it enables sustainable water treatment designs, improving pollutant removal in real-world plants like wastewater facilities by guiding microbial selection and control strategies for cost-effective, ecosystem-friendly operations amid variable inflows and environmental stresses.
\end{abstract}

\begin{keyword}
Bang-bang control \sep Caputo fractional derivative \sep Chemostat model \sep Fractional-order control \sep Memory effects \sep Optimal periodic control \sep Water treatment.\\[0.25em]
\textit{MSC:} 34A08; 37N25; 49J15; 92D25.
\end{keyword}
\end{frontmatter}

\begin{table}[ht]
\centering
\resizebox{0.45\textwidth}{!}{
\begin{tabular}{ll}
\toprule
\textbf{Acronym/Notation} & \textbf{Definition} \\
\midrule
a.e. & Almost everywhere \\
CFDS & Caputo Fractional Derivative with Sliding Memory \\
FD & Fractional Derivative \\
FDE & Fractional-Order Differential Equation \\
FFT & Fast Fourier Transform \\
FG-PS & Fourier--Gegenbauer Pseudospectral \\
FOCS & Fractional-Order Chemostat System \\
FOCM & Fractional-Order Chemostat Model \\
FOCP & Fractional-Order Optimal Control Problem \\
NLP & Nonlinear Programming \\
PBC & Periodic Boundary Condition \\
PMP & Pontryagin's Maximum Principle \\
OC & Optimal Control \\
OOFV & Optimal Objective Function Value \\
OPC & Optimal Periodic Control \\
OPS & Optimal Periodic Solution \\
PSC & Periodic Substrate Concentration \\
RFOCP & Reduced Fractional-Order Optimal Control Problem \\
RL & Riemann--Liouville \\
$\MB{N}$ & Set of positive integer numbers \\
$f\av$ & Average Value of a Periodic Function $f$ Over One Period \\
$f_{\varepsilon,\av}$ & Average Value of a Periodic Function $f_{\varepsilon}$ Over One Period \\
$\xi^*\av$ & Average Value of $\xi^*$ Over One Period \\
$f'$ & The First Derivative of a Function $f$ \\
$\Gamma(\cdot)$ & The Gamma Function \\
$E_\alpha(z)$ & The One-Parameter Mittag-Leffler Function with $\alpha > 0$, Defined by \\
& $\displaystyle{E_\alpha(z) = \sum_{n=0}^\infty \frac{z^n}{\Gamma(n\alpha + 1)}}$ \\
${}^{\text{MC}}_{L}D_t^\alpha f(t)$ & The (left-sided) Caputo Fractional Derivative of the Function $f$ with Sliding Memory, Defined by \\
&$\displaystyle{{}^{\text{MC}}_{L}D_t^\alpha f = \frac{1}{\Gamma(1-\alpha)} \int_{t-L}^t (t-\tau)^{-\alpha} f'(\tau) \, d\tau}$, where $\alpha \in (0, 1)$ is the fractional order and $L > 0$ is the,\\
& sliding memory length \\
${}^{\text{MC}}_{L+}D_t^\alpha f$ & The right-sided Caputo Fractional Derivative of the Function $f$ with Sliding Memory, Defined by \\
&$\displaystyle{{}^{\text{MC}}_{L+}D_t^\alpha f = -\frac{1}{\Gamma(1 - \alpha)} \int_t^{t+L} (\tau - t)^{-\alpha} f'(\tau)\, d\tau}$, where $\alpha \in (0, 1)$ is the fractional order and $L > 0$ is the,\\
& sliding memory length \\
$W^{1,1}_{\text{loc}}([a, b])$ & Sobolev space of functions whose first weak derivative exists and is locally integrable over the interval $[a, b]$\\
$AC_T$ & The absolutely continuous space of $T$-periodic functions with the norm $\|s\|_{AC} = \|s\|_{\infty} + \|s'\|_{L^1}$,\\
& where $\|s\|_{\infty} = \sup_{t \in [0,T]} |s(t)|$ and $\|s'\|_{L^1} = \int_0^T |s'(t)|\,dt$ \\
\bottomrule
\end{tabular}
}
\caption{List of Acronyms and Notations}
\end{table}

\section{Introduction}
\label{sec:Int}
The chemostat is an essential bioreactor in environmental engineering, enabling the cultivation of microorganisms for pollutant degradation and playing a crucial role in biological water treatment processes. Traditionally, chemostats are operated at steady state to ensure predictable and consistent performance in reducing pollutant concentrations. However, recent research has demonstrated that periodic operation, where control inputs such as dilution rates are varied over time, can outperform steady-state operation in both efficiency and pollutant removal, particularly when aligned with the specific growth kinetics of the microorganisms involved \cite{bayen2020improvement,elgindy2025sustainable,elgindy2023new}. This shift towards dynamic operation opens new avenues for optimizing bioprocesses, especially in the context of clean water production and ecosystem health preservation.

Conventional chemostat models rely on integer-order differential equations, which assume instantaneous responses and neglect the memory effects inherent in biological systems. However, numerous studies have demonstrated that biological systems often exhibit long-term memory and hereditary behaviors that are better captured by fractional calculus formulations \cite{tarasov2011fractional,magin2004fractional,caponetto2010fractional,chakraverty2023time}. This framework has been successfully applied to model memory effects in respiratory tissue, drug pharmacokinetics, and neural adaptation, providing a parsimonious and accurate representation of complex, multi-scale biological dynamics \cite{ionescu2017role}. Recent works further demonstrate the effectiveness of fractal-fractional analysis in water pollution modeling, highlighting the growing recognition of fractional approaches in environmental applications \cite{elgindy2025sustainable,loudahi2025fractal}. 

In reality, microbial growth and substrate degradation exhibit time-dependent behaviors influenced by past states, such as delayed responses to nutrient availability or environmental changes. To address this limitation, we propose a FOCM that incorporates memory effects through the CFDS. This approach captures long-term memory and non-local effects, providing a more accurate representation of complex biological interactions than traditional integer-order chemostat models. The sliding memory window introduces a finite memory effect, enabling the model to account for historical states over a bounded time horizon, which is particularly important for systems exhibiting slow adaptation or persistent environmental influence \cite{du2013measuring,ionescu2017role}.

Unlike the classical Caputo and RL FDs, which rely on fixed memory starting from an initial time, the CFDS employs a finite sliding memory window $[t - L, t]$, making it better suited for modeling biological systems with localized historical dependence. This structure preserves periodicity and ensures that the FD vanishes over one full cycle for periodic functions---a vital characteristic for enforcing PBCs as we demonstrate later in \ref{app:LIC1}. The selection of the CFDS is therefore motivated by three primary reasons that stem from this structure: (i) its sliding window $[t-L, t]$ aligns with finite microbial memory in real bioprocesses, avoiding dependence on the entire past history assumptions that can lead to non-physical artifacts in periodic regimes; (ii) it preserves the zero-integral property over periods for periodic functions (see Lemma \ref{lem:1} in \ref{app:LIC1}), ensuring PBC compatibility without additional corrections; and (iii) it enables efficient numerical handling via pseudospectral methods, reducing computational complexity compared to full-memory operators \cite{elgindy2025sustainable}. In contrast, RL FDs require fractional initial conditions that lack clear biological interpretation, while standard Caputo FDs disrupt periodicity due to the fixed integration lower limit. For further information on the CFDS, its properties, and applications, see \cite{elgindy2025sustainable,elgindy2024fouriera,elgindy2024fourierb,elgindy2024fourier} and the references therein. 

In this paper, we focus on optimizing a FOCS for continuous biological water treatment under periodic control. The primary objective is to minimize the average output concentration of the pollutant (substrate) over a fixed period, thereby improving the quality of treated water. This optimization is subject to constraints on the average dilution rate, ensuring a consistent treatment capacity, and PBCs that reflect the cyclic nature of the control strategy. By utilizing FD dynamics, we aim to utilize the memory effects to achieve superior pollutant removal compared to steady-state controls, and more realistic representation of the system dynamics than integer-order periodic controls. 

This paper makes several significant contributions to the field of bioprocess engineering and fractional calculus: (i) We develop a novel FOCM that integrates a CFDS effect, extending the integer-order framework of \cite{bayen2020improvement}. This model captures the memory-dependent dynamics of microbial growth and substrate degradation, offering a more realistic representation of bioprocesses. (ii) In addition to the fractional order $\alpha$, which quantifies microbial memory effects, our model introduces a dynamic scaling parameter $\vartheta > 0$ that modulates the amplitude of system dynamics. This parameter appears as a multiplicative factor $\vartheta^{1-\alpha}$ in the FDEs, scaling the influence of dilution and microbial activity. While $\alpha$ captures the degree of memory, $\vartheta$ controls the system's responsiveness to control inputs. As we demonstrate, tuning $\vartheta$ significantly impacts pollutant removal efficiency, with larger values increasing dynamic response and yielding substantial reductions in average substrate concentration. (iii) We reduce the 2D FOCS to a 1D FDE using a transformation that links substrate and biomass concentrations, as derived in \cite{elgindy2025sustainable}. This reduction results in analytical simplification and computational efficiency while preserving the essential dynamics. (iv) We formulate a FOCP to minimize the average pollutant concentration under OPC, incorporating FD dynamics and practical constraints on dilution rates and treatment capacity. This extends the scope of periodic control strategies to FOCSs. (v) We establish the existence of OPSs using results from \cite{elgindy2025sustainable}, which employ Schauder's fixed-point theorem and compactness arguments. We also provide conditions for the uniqueness of solutions for specific parameter value sets to improve the robustness of the proposed approach. (vi) We conduct a comprehensive sensitivity analysis to quantify the impact of key fractional parameters (specifically the fractional order $\alpha$, the scaling parameter $\vartheta$, and the memory length $L$) on the OPC performance and the resulting average pollutant concentration. This analysis provides crucial, actionable insights for system designers, illustrating how to tune these parameters to maximize treatment efficiency while minimizing environmental impact under varying biological and operational conditions, hopefully resulting in cleaner water and healthier ecosystems.

The motivation for this work stems from the urgent need for more efficient, sustainable water treatment technologies capable of handling variable environmental conditions while minimizing ecological impact. By incorporating microbial memory effects via CFDS, we develop biologically realistic and practically implementable OPC strategies for real-world wastewater treatment facilities. This study advances the theoretical and practical understanding of OPC in FOCSs, bridging fractional-order dynamics with environmental engineering applications, and offers a pathway to improved pollutant removal and ecosystem health preservation.

The remainder of this paper is organized as follows. Section \ref{sec:Prelim} provides essential preliminaries on fractional calculus and the CFDS. Section \ref{sec:PS1} presents the problem statement and formulation of the FOCM. Section \ref{sec:SFOCM1} simplifies the FOCM to a 1D FDE, establishes the existence and uniqueness of its solutions, and analyzes the computational complexity of the reduced system. Section \ref{sec:optimal_control} derives the OC strategy using the fractional PMP. Section \ref{sec:Mem} analyzes the memory effects on control stability and orbital stability. Section \ref{sec:NTP1} presents sensitivity analysis and numerical simulations. Finally, Section \ref{eq:Conc} concludes the paper with a summary of key findings and future research directions. The appendices contain supporting theoretical results, including integral properties of the CFDS (\ref{app:LIC1}), state convexity analysis (\ref{app:convexity}), local stability of equilibria (\ref{app:LSE}), orbital stability of OPSs (\ref{sec:orbital_stability}), perturbation analysis (\ref{app:perturbation_analysis}), derivation of the right-sided CFDS (\ref{sec:DRSCFDS1}), numerical optimization techniques (\ref{app:NSMfStR1}), and the complete solution algorithm (\ref{app:algorithm}).

\section{Preliminaries}
\label{sec:Prelim}
This section provides a brief introduction to the key concepts of fractional calculus employed in this paper, with a focus on the Caputo FD and its sliding memory variant (CFDS). Fractional calculus extends traditional integer-order differentiation and integration to non-integer orders, enabling the modeling of systems with memory and hereditary properties \cite{podlubny1998fractional, kilbas2006theory}. In biological systems like chemostats, these tools capture time-dependent behaviors such as microbial adaptation delays, which integer-order models overlook.

\subsection{FDs and Integrals}
The Riemann-Liouville fractional integral of order $\alpha > 0$ for a function $f(t)$ is defined as
\[
I_a^\alpha f(t) = \frac{1}{\Gamma(\alpha)} \int_a^t (t - \tau)^{\alpha - 1} f(\tau) \, d\tau.
\]
FDs build upon this, with several definitions available (e.g., Riemann-Liouville, Caputo, Gr\"{u}nwald-Letnikov). The left-sided Caputo FD of order $\alpha \in (n-1, n)$ for $n \in \MB{N}$ is
\[
{}^{\text{C}}_a D_t^\alpha f(t) = I_a^{n-\alpha} f^{(n)}(t) = \frac{1}{\Gamma(n-\alpha)} \int_a^t (t - \tau)^{n-\alpha-1} f^{(n)}(\tau) \, d\tau.
\]
For $0 < \alpha < 1$, this simplifies to
\[
{}^{\text{C}}_a D_t^\alpha f(t) = \frac{1}{\Gamma(1-\alpha)} \int_a^t (t - \tau)^{-\alpha} f'(\tau) \, d\tau.
\]
The Caputo derivative is particularly suited for initial value problems, as it incorporates initial conditions in terms of integer-order derivatives, aligning with physical interpretations like the initial substrate concentrations in bioprocesses. Unlike integer-order derivatives, which are local, Caputo FDs are nonlocal operators that depend on the entire history from $a$ to $t$, modeling hereditary phenomena in microbial kinetics \cite{magin2004fractional}.

\subsection{CFDS}
\label{subsec:CFDS1114Oct}
Traditional Caputo FDs use a fixed lower limit $a$ (often 0), implying dependence on the entire past history from the initial time. In periodic bioprocesses with cyclic operations, this can conflict with periodicity and bounded historical dependence such as the microbial responses influenced only by recent environmental changes over a finite window $L$. To address this, we employ the CFDS, a finite-memory variant introduced in \cite{bourafa2021periodic}, and modified later by \citet{elgindy2024fouriera,elgindy2024fourierb} for improved numerical stability. The left-sided CFDS is
\[
{}^{\text{MC}}_{L} D_t^\alpha f(t) = \frac{1}{\Gamma(1-\alpha)} \int_{t-L}^t (t - \tau)^{-\alpha} f'(\tau) \, d\tau,
\]
where $L > 0$ is the sliding memory length (a bounded horizon $[t-L, t]$). This preserves the Caputo kernel but localizes memory to a finite sliding window $[t-L, t]$, making it ideal for modeling systems with hereditary effects over bounded horizons, such as slow microbial adaptation or persistent finite influences in biological tissues and physiological processes \cite{du2013measuring}. In the context of wastewater treatment, this finite memory captures nutrient or pollutant history over practical timescales (like hours of inflow variability resetting microbial responses), avoiding the unphysical full history recall of classical FDs while aligning with observed non-local behaviors in bioprocesses \cite{ionescu2017role}.

The CFDS offers several advantages over standard Caputo or Riemann-Liouville FDs: (i) its sliding window $[t-L, t]$ ensures that for $f \in AC_T$, the integral of the CFDS over $[0,T]$ vanishes (see Lemma \ref{lem:1} in \ref{app:LIC1}), guaranteeing compatibility with the PBCs \eqref{eq:PBC1}--\eqref{eq:PBC3}; (ii) the finite window $L$ confines convolutions to bounded intervals, in contrast to the infinite memory in classical definitions, facilitating efficient pseudospectral discretization and reducing computational complexity \cite{elgindy2024fouriera,elgindy2024fourierb}; (iii) it is more biologically plausible, as the sliding memory mechanism localizes hereditary effects (i.e. captures microbial memory over recent time), thereby preventing the unphysical dependence on the entire distant history that characterizes standard Caputo/RL FDs with fixed lower limits \cite{du2013measuring, ionescu2017role}. For further details on the properties of the CFDS, see \cite{bourafa2021periodic,elgindy2025sustainable} and the references therein.

\section{Problem Statement}
\label{sec:PS1}
In this paper, we address the optimization of a chemostat model for continuous biological water treatment, where we focus on minimizing the average output concentration of pollution under periodic control strategies. Our primary goal is to minimize the average output concentration of the pollutant (substrate), denoted $s(t)$, over a fixed period $T$. This translates to the practical goal of reducing pollutant levels in the effluent of a water treatment process. The total amount of water treated during the period $T$ is required to have an average removal rate $\bD$, calculated by dividing the total treated volume $\bar{Q}$ by the product of the chemostat volume $V$ and the period $T$, i.e., $\bD = \bar{Q}/(V T)$. This constraint ensures a consistent treatment capacity by explicitly quantifying the total treated volume, modeling practical limitations such as reactor volume or pump rates in wastewater facilities, thereby ensuring operational feasibility. The FOCP is formulated as follows:

\begin{subequations}
\begin{equation}\label{eq:1a}
\min_{D} J(D),
\end{equation}
subject to the integral constraint on the control variable $D(t)$, the dilution rate:
\begin{equation}\label{eq:IntConsD1}
D\av = \bD,
\end{equation}
the state and control bounds 
\begin{gather}
0 \le s(t) \le \Sin,\label{eq:Boundd1}\\
x(t) > 0,\label{eq:Boundd2}\\
D_{\min} \le D(t) \le D_{\max},\label{eq:Boundd3}
\end{gather}
and the following 2D FDEs expressing the system dynamics:
\begin{align}
{}^{\text{MC}}_{L}D_t^\alpha s(t) &= \vartheta^{1-\alpha} \left[-\frac{1}{Y} \mu(s(t), x(t)) x(t) + D(t) (s_{in} - s(t))\right],\label{eq:sysdyn1}\\
{}^{\text{MC}}_{L}D_t^\alpha x(t) &= \vartheta^{1-\alpha}\,[\mu(s(t), x(t)) - D(t)]\,x(t),\label{eq:sysdyn2}
\end{align}
with PBCs:
\begin{align}
s(t) &= s(t+T),\quad \forall t \in [0, \infty),\label{eq:PBC1}\\
x(t) &= x(t+T),\quad \forall t \in [0, \infty),\label{eq:PBC2}\\
D(t) &= D(t+T),\quad \forall t \in [0, \infty).\label{eq:PBC3}
\end{align}
\end{subequations}
In the above FOCP, $s(t)$ represents the substrate concentration, $x(t)$ the biomass concentration, $D(t)$ the time-varying dilution rate (control input) with minimum and maximum values $D_{\min}$ and $D_{\max}$, respectively, $J(D) = s\av$ is the objective functional, which represents the average substrate concentration over the period $T$, $s_{in}$ the inlet substrate concentration, $Y > 0$ the yield coefficient, $\mu(s(t), x(t))$ the specific growth rate of the microorganisms, and $\vartheta > 0$ is a dynamic scaling parameter that controls the magnitude of the system dynamics and ensures dimensional consistency in the fractional-order equations. The specific growth rate $\mu$ is assumed to follow the Contois growth model given by:
\begin{equation}\label{eq:muContois1}
\mu(s, x) = \frac{\mu_{\max} s}{K x + s},    
\end{equation}
where $\mu_{\max} > 0$ is the maximum growth rate, and $K > 0$ is the saturation constant. The fractional dynamics are modeled using the CFDS. The sliding memory window $[t-L, t]$ introduces a finite memory effect, which captures the influence of past states on current dynamics. This formulation extends the integer-order chemostat model studied in \cite{bayen2020improvement} by accounting for memory effects, which can lead to a more realistic representation of microbial growth and substrate degradation. The challenge lies in determining the OPC $D^*$ that minimizes the average substrate concentration while satisfying the treatment constraint and periodic conditions in this fractional-order context. 

\subsection{Biological Interpretation of CFDS and Fractional Parameters}
\label{subsec:Biommmn1}
The CFDS is biologically motivated, not merely mathematical. In real FOCSs:
\begin{itemize}
    \item The microbial memory is finite and local \cite{gokhale2021memory,faigenbaum2025uncovering}. In fact, bacteria respond to nutrient and stress history over hours to days, not since reactor startup. The CFDS window $[t-L,t]$ in ${}^{\text{MC}}_{L} D_t^\alpha$ enforces this bounded memory, avoiding unphysical infinite recall.
    \item Fractional order $\alpha$ quantifies memory strength \cite{du2013measuring,khalighi2022quantifying,amirian2022extending}:
    \begin{itemize}
        \item $\alpha \to 1$: Weak memory (short-term memory, fast-growing strains).
        \item $\alpha \to 0$: Strong memory (persistent historical dependence like biofilm populations \cite{rumbaugh2020biofilm}).
    \end{itemize}
    Lower $\alpha$ increases the weight of past states in Eqs. \eqref{eq:sysdyn1} and \eqref{eq:sysdyn2}, requiring more dynamic OPC to overcome inertia, as we demonstrate later in Section \ref{sec:NTP1}.
    \item Memory length $L$ reflects ecological timescale; for example, $L$ may equal the hydraulic residence time \cite{raboni2013influence}.
    \item Scaling parameter $\vartheta$ controls responsiveness: Larger $\vartheta$ amplifies dilution and growth terms. Tuning $\vartheta$ enables strain-specific OPC like, for example, selecting high-$\vartheta$ strains for unstable inflows.
\end{itemize}

This framework predicts how memory-aware OPC outperforms steady-state by aligning control with microbial timescales, enabling smarter bioprocess engineering.

\section{Simplification of the FOCM}
\label{sec:SFOCM1}
To facilitate analysis and computation, we can reduce the FOCM, characterized by the 2D FDEs \eqref{eq:sysdyn1}--\eqref{eq:sysdyn2}, to a 1D FDE. As shown in \cite[Section 2.2]{elgindy2025sustainable}, the 2D FOCS can be transformed into a 1D FDE using a transformation that uses PBCs and the properties of the CFDS. In particular, the transformation 
\begin{equation}\label{eq:Transs1}
z(t) = Y (\Sin - s(t)) - x(t),
\end{equation}
applied to the FOCS \eqref{eq:sysdyn1}--\eqref{eq:sysdyn2} results in the FDE:
\begin{equation}\label{eq:trans2}
{}_L^{\text{MC}} D_t^\alpha z(t) = -\vartheta^{1-\alpha} D(t) z(t),
\end{equation}
with PBC $z(t) = z(t+T)$, which follows from $s(t) = s(t+T)$ and $x(t) = x(t+T)$. Using an energy dissipation argument, it can be shown that the FDE \eqref{eq:trans2} admits no nontrivial periodic solutions under the specified dynamics and boundary conditions. Consequently, we have $z(t) \equiv 0$, and hence
\begin{equation}\label{eq:biomass2}
x(t) = Y (\Sin - s(t)).
\end{equation}
Using the relation \eqref{eq:biomass2}, the FDEs \eqref{eq:sysdyn1}--\eqref{eq:sysdyn2} are reduced to the following 1D FDE:
\begin{equation}\label{eq:reducedFDE}
{}_L^{\text{MC}} D_t^\alpha s(t) = \C{F}(t, s(t)),
\end{equation}
subject to the PBC \eqref{eq:PBC1}, where 
\begin{equation}\label{eq:RHS1sds}
\C{F}(t, s(t)) = \vartheta^{1-\alpha} [D(t) - \nu(s(t))] (\Sin - s(t)), 
\end{equation}
and 
\begin{equation}\label{eq:nu1}
\nu(s(t)) = \mu(s(t), Y (\Sin - s(t))) = \frac{\mu_{\max} s(t)}{K Y (\Sin - s(t)) + s(t)}.
\end{equation}
Here, $\nu(s(t))$ represents ``the substrate-dependent specific growth rate.'' This FDE governs the substrate concentration $s$, with $D$ as the control input, enabling optimization of the objective \eqref{eq:1a} under constraints \eqref{eq:IntConsD1}, \eqref{eq:Boundd1}, \eqref{eq:Boundd3}, \eqref{eq:PBC1}, and \eqref{eq:PBC3}. We refer to this reduced FOCP by the RFOCP.

For a constant dilution rate $D \equiv \bD$ and constant substrate concentration $s \equiv \bs$, the nontrivial equilibrium solution of \eqref{eq:reducedFDE} is given by \cite[Section 2.3]{elgindy2025sustainable}:
\begin{equation}\label{eq:equil2}
\bs = \frac{\bD K Y \Sin}{\bD K Y + \mu_{\max} - \bD},
\end{equation}
provided $\bD < \mu_{\max}$, ensuring $\bs < \Sin$ and a positive biomass concentration via Eq. \eqref{eq:biomass2}. This equilibrium satisfies $\nu(\bs) = \bD$.

\subsection{Computational Complexity Analysis}
\label{subsec:complexity}
The computational complexity of solving the reduced FOCS \eqref{eq:reducedFDE} is primarily governed by the evaluation of the CFDS. As established in our previous work \cite{elgindy2024fouriera}, the CFDS operator ${}^{\text{MC}}_{L}D_t^\alpha s(t)$ introduces computational challenges due to its non-local nature, requiring integration over the sliding memory window $[t-L, t]$ at each time step.

For a numerical solution discretized over $N$ time steps with step size $\Delta t = T/N$, the memory length $L$ covers $M = \lfloor L/\Delta t \rfloor = \lfloor LN/T \rfloor$ previous steps. For any general time step $t_i$ with $i \ge M$, a naive implementation evaluating the CFDS via direct quadrature summation requires $\C{O}(M)$ operations per time step, as it must integrate over the entire memory window $[t_i-L, t_i]$. Over $N$ total time steps, this results in $\C{O}(N M) = \C{O}(N^2)$ computational complexity in the worst case, when $L$ is comparable to the period $T$. However, our FG-PS method \cite{elgindy2024fouriera} significantly reduces this complexity through several key innovations:
\begin{enumerate}[label=(\roman*)]
\item The CFDS operator ${}^{\text{MC}}_{L}D_t^\alpha s(t)$ is discretized using the $\alpha$-th order FG-PS integration matrix (FGPSIM) developed in \cite{elgindy2024fouriera}, which serves as a discrete fractional differentiation operator. Due to the convolutional nature of the sliding-memory kernel (Theorem 4.1 in \cite{elgindy2024fouriera}), this matrix exhibits a Toeplitz structure that reduces storage from $\C{O}(N^2)$ to $\C{O}(N)$ and, combined with the periodic nature of the problem, enables matrix–vector products in $\C{O}(N \log N)$ time via FFT-based convolution. This significantly improves overall scalability for large-scale and real-time bioprocess optimization.
	\item Unlike classical FDs with infinite memory, the CFDS employs a finite memory length $L$, bounding the effective historical dependence and preventing unbounded computational growth.
    \item The FG-PS method exhibits spectral convergence (Theorem 5.2 in \cite{elgindy2024fourier}), meaning that moderate values of $N$ (typically 300--400 in our simulations) suffice for engineering accuracy, keeping computational requirements manageable.
\end{enumerate}
The overall computational complexity of solving the reduced 1D system is dominated by the FG-PS discretization, which requires $\C{O}(N \log N)$ operations for the FD evaluations, plus the cost of solving the resulting NLP problem.

The reduction from the original 2D system \eqref{eq:sysdyn1}--\eqref{eq:sysdyn2} to the 1D formulation \eqref{eq:reducedFDE} provides substantial computational advantages by reducing the number of state variables from two ($s(t)$ and $x(t)$) to one ($s(t)$ only). This simplification decreases both the dimensionality of the discretized system and the complexity of the resulting NLP problem, leading to faster computation times and reduced memory requirements compared to solving the full coupled system.

\subsection{Existence of Solutions to the RFOCP}
\label{subsec:existence}
Define 
\begin{equation}
X = \{s \in AC_T \mid 0 \leq s(t) \leq \Sin\}.
\end{equation}
Then $X$ is a compact and convex subset of $AC_T$ representing the set of feasible substrate concentrations. We also take the admissible control set to be:
\begin{align}\label{eq:OCfset1}
\C{D} = \{D \in L^\infty([0, T]) |\, & D\text{ satisfies } \eqref{eq:IntConsD1}, \eqref{eq:Boundd3},\text{ and }\eqref{eq:PBC3}\}.
\end{align}
The following theorem uses results from \cite{elgindy2025sustainable} to establish the existence of an OPC $D^*$ and its corresponding state $s^*$ for the RFOCP.

\begin{theorem}[Existence of OPC]
\label{thm:existence}
Suppose that $s(0) \in (0, \Sin)$, $x(0) > 0$, and $D(t) < \mu_{\max}$ for all $t \geq 0$. Then, the RFOCP admits at least one OPS $(s^*, D^*) \in X \times \C{D}$ satisfying $s^* < \Sin$.
\end{theorem}
\begin{proof}
Notice first that the set $\C{D}$ is non-empty, since $D \equiv \bD \in L^\infty([0, T])$ satisfies \eqref{eq:IntConsD1}, \eqref{eq:Boundd3}, and \eqref{eq:PBC3}. Theorem 2.2 in \cite{elgindy2025sustainable} guarantees the existence of a nontrivial, $T$-periodic Carath\'{e}odory solution to the FDE \eqref{eq:reducedFDE} with $s < \Sin$, for any admissible control $D \in \C{D}$. This ensures that the dynamics are well-defined. Since $s$ is absolutely continuous and the control-to-state mapping $\C{T}: \C{D} \to X$ is continuous (by \cite[Lemma 2.3]{elgindy2025sustainable}), the objective functional $J(D)$ is continuous. Notice also that $\C{D}$ is norm-bounded in $L^\infty([0, T])$ by \eqref{eq:Boundd3} and weakly-$\ast$ closed, since the integral constraint \eqref{eq:IntConsD1} is weakly-$\ast$ continuous, and the uniform bounds \eqref{eq:Boundd3} and periodicity \eqref{eq:PBC3} are preserved under weak-$\ast$ convergence. Thus, by the Banach-Alaoglu theorem, $\C{D}$ is weakly-$\ast$ compact \cite{clarke2013functional}. Moreover, for any sequence of controls $\{D_n\} \in \C{D}$ converging weakly-$\ast$ to $D$, the corresponding solutions $s_n \to s \in X$ (by compactness of $X$). By \cite[Lemma 2.3]{elgindy2025sustainable}, $s$ solves \eqref{eq:reducedFDE} for $D$, as the solution of the Volterra integral equation \cite[Eq. (22)]{elgindy2025sustainable} converges to the solution of the reduced fractional chemostat equation \eqref{eq:reducedFDE}. The weak-$\ast$ compactness of $\C{D}$, continuity of $J$, and compactness of $X$ ensure the infimum of \eqref{eq:1a} is attained at $(s^*, D^*)$ \cite{ahmed2021optimal}.
\end{proof}

\begin{remark}
The condition $D(t) < \mu_{\max}$ in Theorem \ref{thm:existence} ensures the dilution rate does not exceed the maximum growth rate, preventing washout, where biomass is flushed out faster than it grows.
\end{remark}

Having established the existence of OPCs, we now prove the possible existence of non-constant OPCs for the RFOCP. To prove Theorem \ref{thm:nonconstant_existence} below, we introduce the assumption below.

\begin{assumption}\label{Assump:1}
For sufficiently small $\varepsilon$, there exists a perturbed control $D_\varepsilon(t) = \bD + \varepsilon v(t)$, where $v(t)$ is $T$-periodic with $v\av = 0$, ensuring $D\epsav = \bD$, with the corresponding state having the form $\seps(t) = \bs + \varepsilon z(t)$, with $z(t)$ $T$-periodic.
\end{assumption}

Notice that, due to the well-posedness of the FDE and the continuity of the control-to-state mapping, Assumption \ref{Assump:1} is always valid under the conditions of the RFOCP, provided $\varepsilon$ is sufficiently small, as established in \cite{elgindy2025sustainable}.

While Theorem \ref{thm:existence} guarantees the existence of an OPS pair $(s^*, D^*)$ for the RFOCP, it does not determine whether the OPC $D^*$ must be constant or if non-constant solutions are possible. A natural question arises: Are all possible OPCs necessarily constant, or can non-constant controls yield better performance? Theorem \ref{thm:nonconstant_existence} addresses this critical gap by proving that, under specific conditions, non-constant OPCs may indeed exist and can achieve superior performance compared to steady-state solutions. This result underscores the potential advantages of periodic control strategies, particularly when accounting for memory effects and dynamic scaling in fractional-order systems.

\begin{theorem}[Possible Existence of Non-Constant OPCs]
\label{thm:nonconstant_existence}
Let $KY \neq 1, \alpha \in (0,1)$, and suppose that the conditions of Theorem \ref{thm:existence} are satisfied. Then, the RFOCP may admit an OPS $(s^*, D^*) \in X \times \C{D}$, where $D^*$ is non-constant, and the corresponding non-constant state $s^*$ satisfies $s^*\av < \bs$, with the potential to improve upon the steady-state average substrate concentration.
\end{theorem}
\begin{proof}
By Theorem \ref{thm:existence}, the RFOCP admits an optimal solution $(s^*, D^*) \in X \times \C{D}$ with $s^* < \Sin$. Following \cite[Lemma 2]{bayen2020improvement}, consider a $T $-periodic, measurable function $v(t)$ that is non-zero a.e. with $v\av = 0$. Define the control $D_\varepsilon(t) = \bD + \varepsilon v(t)$, with $\varepsilon > 0$ small enough that $D_{\min} \leq D_\varepsilon(t) \leq D_{\max}$. Since $v\av = 0 $, we have $D\epsav = \bD $, so $D_\varepsilon \in \C{D}$. Define the mapping:
\[
\theta(s_0, \varepsilon) = s(T, D_\varepsilon, s_0) - s_0,
\]
where $s(t, D_\varepsilon, s_0)$ is the solution to the FDE \eqref{eq:reducedFDE} with control $D_\varepsilon$ and initial condition $s(0) = s_0$. By \cite[Lemma 2.3]{elgindy2025sustainable}, the control-to-state mapping $\C{T}: \C{D} \to X $ is continuous. Since $\C{F}$ is Lipschitz in $s$ \cite[Lemma 2.1]{elgindy2025sustainable}, the solution $s$ is continuous in $s_0$. Thus, $\theta(s_0, \varepsilon)$ is continuous in $s_0$ and $\varepsilon$. For $\varepsilon = 0 $, $D_0 \equiv \bD $, and $s \equiv \bs $, so $\theta(\bs, 0) = 0$. In accordance with Theorem \ref{thm:local_stability} (referenced in \ref{app:LSE}), the behavior of $s(t, \bD, s_0)$ can be described as follows:
\begin{itemize}
    \item If $s_0^- < \bs$, then $s(t, \bD, s_0^-)$ will increase asymptotically towards $\bs$.
    \item Conversely, if $s_0^+ > \bs$, then $s(t, \bD, s_0^+)$ will decrease asymptotically towards $\bs$.
\end{itemize}
Thus:
\[
\theta(s_0^-, 0) > 0, \quad \theta(s_0^+, 0) < 0.
\]
For small $\varepsilon$, continuity ensures $\theta(s_0^-, \varepsilon) > 0$, $\theta(s_0^+, \varepsilon) < 0$, so there exists $\bs_0 \in (s_0^-, s_0^+)$ such that $\theta(\bs_0, \varepsilon) = 0$, giving us a $T$-periodic, non-constant solution $s_\varepsilon(t)$. 

For improvement, suppose that $K Y < 1$. In this case, $\nu$ is strictly concave, so:
\begin{equation}\label{eq:23July2025}
\nu(\seps(t)) < \nu(\bs) + \nu'(\bs)(\seps(t) - \bs),\quad a.e.\,t,
\end{equation}
where
\begin{equation}\label{eq:nuprime1}
    \nu'(s) = \frac{K Y \mu_{\max} \Sin}{(K Y (\Sin - s) + s)^2}.    
\end{equation}
Taking the time average of both sides of Inequality \eqref{eq:23July2025} gives:
\begin{equation}
[\nu(\seps)]_{av} < \nu(\bs) + \nu'(\bs)(s\epsav - \bs),\quad a.e.\,t.
\end{equation}
Rearranging the Inequality:
\begin{equation}
[\nu(\seps)]_{av} - \nu(\bs) < \nu'(\bs)(s\epsav - \bs).
\end{equation}
From Assumption \ref{Assump:1} and Eq. \eqref{eq:ssvegdflkl1} in \ref{app:perturbation_analysis}, we have $[\nu(\seps)]_{av} < \bD = \nu(\bs)$. Therefore, $[\nu(\seps)]\av - \nu(\bs) < 0$. Since $\nu'$ is always positive, we have $\nu'(\bs) > 0$. The inequality's negative right-hand side, arising when \begin{equation}\label{eq:weyewt1}
s\epsav < \bs,
\end{equation}
confirms that non-constant OPCs can reduce the average substrate concentration below the steady-state level. To support this claim further, notice by Jensen's inequality that
\begin{equation}
\nu(s\epsav) > [\nu(\seps)]\av.
\end{equation}
However, $[\nu(\seps)]_{av} < \bD = \nu(\bs)$, by Eq. \eqref{eq:ssvegdflkl1}, so
\begin{equation}
\nu(s\epsav) > [\nu(\seps)]\av < \nu(\bs).
\end{equation}
This suggests that $\nu(s\epsav) < \nu(\bs) \Leftrightarrow s\epsav < \bs$ may take place for some non-constant, $T$-periodic states, but it is not guaranteed for all. 

Now, suppose that $K Y > 1$. In this case, $\nu$ is strictly convex, so:
\begin{equation}
\nu(\seps(t)) > \nu(\bs) + \nu'(\bs)(\seps(t) - \bs),\quad a.e.\,t.
\end{equation}
Take the time average of both sides:
\begin{equation}
[\nu(\seps)]_{av} > \nu(\bs) + \nu'(\bs)(s\epsav - \bs),\quad a.e.\,t.
\end{equation}
Rearranging the Inequality:
\begin{equation}
[\nu(\seps)]_{av} - \nu(\bs) > \nu'(\bs)(s\epsav - \bs).
\end{equation}
From Assumption \ref{Assump:1} and Eq. \eqref{eq:dskvbwk11121} in \ref{app:perturbation_analysis}, we have $[\nu(\seps)]_{av} > \bD = \nu(\bs)$. Therefore, $[\nu(\seps)]\av - \nu(\bs) > 0$. Since $\nu'(\bs) > 0$, the fact that a negative value on the right-hand side (which occurs if $s\epsav - \bs < 0$) is consistent with the inequality means that $s\epsav - \bs < 0$ is a possible outcome, and so Eq. \eqref{eq:weyewt1} may take place for some non-constant, $T$-periodic states. By another similar argument to the former case, notice by Jensen's inequality that
\begin{equation}
\nu(s\epsav) < [\nu(\seps)]\av.
\end{equation}
However, $[\nu(\seps)]_{av} > \bD = \nu(\bs)$, by Eq. \eqref{eq:ssvegdflkl1}, so
\begin{equation}
\nu(s\epsav) < [\nu(\seps)]\av > \nu(\bs).
\end{equation}
This suggests that for certain non-constant perturbations $v(t)$, we may have $s\epsav < \bs$, though this improvement is not guaranteed for all possible perturbations. 

Suppose now that $K Y = 1$. In this case, the substrate-dependent specific growth rate is linear: $\nu(s) = \mu_{\max} s / \Sin$, with $\nu'(s) = \mu_{\max} / \Sin > 0$ and $\nu''(s) = 0$. The steady-state $\bs$ satisfies $\nu(\bs) = \bD$, so $\bs = \bD \Sin / \mu_{\max}$. Using Assumption \ref{Assump:1}, the perturbation analysis in \ref{app:perturbation_analysis} yields:
\[
[\nu(\seps)]\av = D\epsav = \bD = \nu(\bs).
\]
Thus:
\[
\frac{\mu_{\max}}{\Sin} s\epsav = \frac{\mu_{\max}}{\Sin} \bs \implies s\epsav = \bs.
\]
This shows that the average substrate concentration under small, non-constant perturbations equals the steady-state value, implying no improvement over the steady-state. 
\end{proof}

\begin{remark}
While any non-constant, admissible solution improves the performance index compared to the steady-state solution for $\alpha = 1$, as proven in \cite{bayen2020improvement}, such improvement is not guaranteed for $0 < \alpha < 1$. However, under the condition $KY \neq 1$, there may exist non-constant, admissible solutions that yield an improvement, as demonstrated by Theorem \ref{thm:nonconstant_existence}.
\end{remark}

\subsection{Positivity and Boundedness of OPSs}
\label{subsec:PB1}
The following corollary establishes the positivity and boundedness of solutions under mild conditions. Their proofs can be found in \cite[Theorem 2.1 and Corollary 2.2]{elgindy2025sustainable}. 

\begin{corollary}\label{Cor:Oct14}
Let $D(t)$ be any admissible control for all $t \ge 0$, and suppose that $s(0) \in (0, \Sin)$ and $x(0) > 0$. Then the OPSs of the RFOCP satisfy the following properties:
\begin{enumerate}[label=(\roman*)]
    \item The biomass concentration $x^*(t)$ and PSC $s^*(t)$ remain strictly positive for all $t > 0$, i.e., $x^*(t) > 0$ and $s^*(t) > 0$.
    \item The PSC $s^*(t)$ satisfies $0 < s^*(t) < \Sin$ for all $t > 0$.
\end{enumerate}
\end{corollary}

\subsection{Uniqueness of Solutions to the RFOCP}
\label{subsec:uniqueness}
The uniqueness of the RFOCP depends on two main factors: The uniqueness of the state solution for a given control input, and the convexity properties of the objective function and the system dynamics. In this section, we use conditions from \cite{elgindy2025sustainable} to establish uniqueness.

\begin{theorem}[Uniqueness of OPC]\label{thm:uniqueness}
Let $K Y \neq 1$, and suppose that the conditions of \cite[Theorem 2.3(ii)]{elgindy2025sustainable} hold true. Specifically: 
\begin{equation}
s(0) \leq \hs = \frac{\Sin \sqrt{K Y}}{\sqrt{K Y} + 1}, 
\end{equation}
and either $D(t) \leq \nu(\hs)$ for all $t \in [0, T]$ or $\bD \leq \nu(\hs)$. Then the optimal solution $(D^*, s^*)$ is unique. Furthermore, both $s^*$ and $D^*$ must be non-constant, and the strict convexity of $J$ ensures improved performance over the steady-state.
\end{theorem}
\begin{proof}
\cite[Theorem 2.3(ii)]{elgindy2025sustainable} ensures a unique nontrivial, $T$-periodic, Carath\'{e}odory solution to \eqref{eq:reducedFDE}. This establishes that for any given admissible control $D$, there is a unique corresponding state trajectory $s$. Thus, the control-to-state mapping $\C{T}: \C{D} \to X$ is well-defined and single-valued. The uniqueness of the OPC $D^*$ is closely related to the convexity of the problem. The admissible control set $\C{D}$ is convex, as it is defined by linear constraints. To show that $J$ is strictly convex, consider two distinct controls $D_1, D_2 \in \C{D}$ with corresponding states $s_1 = \C{T}(D_1)$, $s_2 = \C{T}(D_2)$. Let $D_\lambda = \lambda D_1 + (1-\lambda) D_2$ for $\lambda \in (0, 1)$, with state $s_\lambda = \C{T}(D_\lambda)$. We need to prove that:
\begin{equation}
J(D_\lambda) = s_{\lambda,\text{av}} < \lambda s_{1,\text{av}} + (1-\lambda) s_{2,\text{av}},
\end{equation}
unless $D_1 = D_2$. Define 
\begin{equation}\label{eq:hofs1}
h(s) = \nu(s) (\Sin - s).
\end{equation}
\cite[Theorem 2.3]{elgindy2025sustainable} shows that $h$ is strictly increasing on $[0, \hs)$, with $h'(\hs) = 0$; moreover, since $D(t) \le \nu(\hs)$ for all $t \in [0,T]$, then $s \in [0, \hs]$, i.e., it remains in the region where $h$ is increasing. Now, define 
\begin{equation}
F(s, D) := D(t)(\Sin - s(t)) - h(s(t)),\quad \forall t \in [0, T].
\end{equation}
Since $D > 0$, we have:
\[
\frac{\partial F}{\partial s} = -D(t) - h'(s) < 0,
\]
so $F$ is strictly decreasing in $s$. This implies the map $D \mapsto s$ is injective, where each admissible control yields a unique state trajectory. The term $h(s)$ introduces nonlinearity in the dynamics. Consequently, for two distinct, admissible controls $D_1$ and $D_2$, the control
\begin{equation}\label{eq:Convexx1}
D_\lambda = \lambda D_1 + (1 - \lambda) D_2,\quad\text{for some }\lambda \in (0,1),
\end{equation}
has a corresponding state $s_\lambda$ that satisfies the nonlinear FDE \eqref{eq:reducedFDE}, and cannot be expressed as a convex combination of $s_1$ and $s_2$ by Lemma \ref{lem:state_convexity}. Thus, the control-to-state map is not affine. Since the control-to-state map $\C{T}$ is nonlinear and injective that does not preserve convex combinations:
\[
  \C{T}(\lambda D_1 + (1-\lambda) D_2) \neq \lambda \C{T}(D_1) + (1-\lambda) \C{T}(D_2),
\]
for any $D_1 \neq D_2$ and $\lambda \in (0,1)$, and the cost functional $J(D)$ is a linear operator applied to the state, the composition $J(D) = J \circ \C{T}(D)$ is strictly convex over the convex, admissible control set $\C{D}$. The OC analysis conducted in Section \ref{sec:optimal_control} reinforces this conclusion by showing that the Hamiltonian is linear in $D$ and admits no singular arcs---a hallmark of strictly convex problems---and the OPC is bang-bang. If the composition $J \circ \C{T}$ were not strictly convex, the Hamiltonian could admit non-bang-bang solutions. The exclusive bang-bang behavior thus confirms that the control-to-state map $\C{T}$ enforces ``a corner solution,'' which is typical of strictly convex optimization problems with linear controls \cite{liberzon2011calculus}. Therefore,
\[
  J(D_\lambda) < \lambda J(D_1) + (1 - \lambda) J(D_2),
\]
for any $D_1 \neq D_2$ and $\lambda \in (0,1)$. Hence, there exists a unique minimizer $D^*$ with corresponding unique optimal trajectory $s^*$. Corollary 2.3 in \cite{elgindy2025sustainable} and the equilibrium definition \eqref{eq:equil2} confirm that the state solution $s$ must be constant when the control $D$ is constant, as optimal constant controls trivially maintain steady-state conditions. However, the same corollary also shows that $s$ must be non-constant when $D$ is non-constant. By Theorem \ref{thm:nonconstant_existence}, there may exist non-constant, admissible solutions that improve upon the steady-state. However, the variability of the unique optimal periodic pair $(s^*,D^*)$ follows from the PMP analysis in Section \ref{sec:optimal_control}, which rules out singular arcs and ensures that the OC must be bang-bang.
\end{proof}

\section{OC Analysis}
\label{sec:optimal_control}
In this section, we derive the OC strategy for the RFOCP using the fractional PMP for the CFDS. For more information on the fractional PMP, readers may consult \cite{agrawal2004general,kamocki2014pontryagin}.

To derive the necessary conditions of optimality, consider the Hamiltonian of the RFOCP:
\[
H(s, p, D) = \frac{1}{T} s(t) + p \vartheta^{1-\alpha} [D(t) - \nu(s(t))] (\Sin - s(t)),
\]
where $p(t)$ is the co-state/adjoint variable. The four PMP conditions are:
\begin{enumerate}[label=(\roman*)]
    \item The system dynamics is recovered from the Hamiltonian:
    \[
    {}_L^{\text{MC}} D_t^\alpha s = \frac{\partial H}{\partial p} = \vartheta^{1-\alpha} [D(t) - \nu(s(t))] (\Sin - s(t)).
    \]
    \item The co-state variable evolves according to:
    \begin{align*}
    {{}^{\text{MC}}_{L+}D_t^\alpha p} &= -\frac{\partial H}{\partial s} = -\frac{1}{T} + p(t) \vartheta^{1-\alpha} \left[\nu'(s(t)) (\Sin - s(t))\right.\\
    &\left.+ D(t) - \nu(s(t)) \right],    
    \end{align*}
    where
    \[
    \nu'(s) = \frac{K Y \mu_{\max} \Sin}{(K Y (\Sin - s) + s)^2}.
    \]
    The use of the right-sided CFDS here reflects the backward-in-time nature of the adjoint system.
    \item For all $t \in [0, T]$, the OC $D^*(t)$ must minimize the Hamiltonian:
    \[
    D^*(t) = \arg \min_{D \in [D_{\min}, D_{\max}]} H(s(t), p(t), D(t)).
    \]
    \item The transversality condition for the co-state must hold:
    \[
    p(0) = p(T).
    \]
\end{enumerate}
Notice that the Hamiltonian is linear in $D$:
\begin{align}
H(s, p, D) &= \left[p(t) \vartheta^{1-\alpha} (\Sin - s(t))\right] D(t) \notag\\
&+ \frac{1}{T} s(t) - p(t) \vartheta^{1-\alpha} \nu(s(t)) (\Sin - s(t)),
\end{align}
so we can define the switching function as follows:
\[
\phi(t) = p(t) \vartheta^{1-\alpha} (\Sin - s(t)).
\]
Since $s(t) < \Sin$ and $\vartheta^{1-\alpha} > 0$, the Hamiltonian is minimized when:
\[
D^*(t) = \begin{cases} 
D_{\max} & \text{if } \phi(t) < 0, \\
D_{\min} & \text{if } \phi(t) > 0, \\
\text{undefined} & \text{if } \phi(t) = 0.
\end{cases}
\]
If $\phi(t) = 0$, then $p(t) = 0$ (since $\Sin - s(t) > 0$). Substituting $p(t) = 0$ into the co-state equation yields:
\[
{{}^{\text{MC}}_{L+}D_t^\alpha p} = -\frac{1}{T},
\]
which gives a contradiction, as the right-sided CFDS of a zero function cannot equal a non-zero constant. Thus, singular arcs are not possible, and the POC $D^*$ is bang-bang, switching between $D_{\min}$ and $D_{\max}$. Due to periodicity, the number of switches per period is even. The switching times are typically computed numerically due to the fractional dynamics.

\section{Memory Effects on Control Stability and Orbital Stability}
\label{sec:Mem}
The fractional order $\alpha$ significantly influences the stability characteristics of the periodic control scheme. As established in Lemma \ref{lem:etey1}, the decay rate $\lambda$ in the exponential stability $z(t) \sim e^{-\lambda t}$ depends on both $\alpha$ and the memory length $L$. Lower values of $\alpha$ (stronger memory effects) result in:
\begin{enumerate}[label=(\roman*)]
    \item More frequent control adjustments to overcome system inertia, as evidenced by the increased number of control switches in the OPS (demonstrated numerically in Section \ref{sec:NTP1}, Figure \ref{fig:Fig5}).
    \item Improved robustness to high-frequency disturbances due to the smoothing effect of historical dependence. This robustness stems from the fundamental difference between integer-order and fractional-order dynamics. In fact, for integer-order dynamics, the derivative depends only on the instantaneous rate of change with high-frequency noise directly affecting the control decisions; hence, the system may overreact to temporary fluctuations. On the other hand, a single noisy measurement in a fractional-order dynamics has a reduced impact because it is averaged with previous states. The control responds to sustained trends rather than momentary spikes. This reduced sensitivity to abrupt control changes, as the sliding memory window averages past states, enables the system to distinguish between true process changes and temporary disturbances.
\end{enumerate}

The above arguments demonstrate that FDs introduce memory effects that contribute to system stabilization. The damping properties established in Theorem \ref{thm:local_stability} for constant dilution rates provide the foundation for analyzing orbital stability under bang-bang OPC. As established in \ref{sec:orbital_stability}, the OPC $D^*$ is bang-bang, and the piecewise-constant structure enables stability analysis on each time segment where Theorem \ref{thm:local_stability} ensures local exponential decay of perturbations. 

The orbital stability framework developed in \ref{sec:orbital_stability} demonstrates that the Poincar\'e map $\C{P}$ exhibits contraction properties over each control period. Specifically, Theorem \ref{thm:jacobian-bound} establishes that for sufficiently small initial perturbation $\delta$, the mapping satisfies $|\C{P}(s^*(0) + \delta) - s^*(0)| \leq \rho\|\delta\|$ with contraction factor $\rho < 1$. This analytical result, combined with the finite number of control switches and the memory-dependent damping characterized in Lemma \ref{lem:etey1}, provides a theoretical foundation for the orbital asymptotic stability of the OPS $s^*(t)$. 

\subsection{Numerical Stability Verification}
\label{subsec:dsndfjk1}
Our comprehensive numerical simulations in Section \ref{sec:NTP1} provide empirical validation of the periodic control scheme's stability:

\begin{itemize}
    \item Figures \ref{fig:Fig1} and \ref{fig:Fig2} demonstrate the consistent convergence to periodic solutions across different discretization levels.
    \item The sensitivity analysis in Figures \ref{fig:Fig5}-\ref{fig:Fig9} shows stable system behavior across wide parameter ranges.
    \item The convergence of multiple initial guesses to either the optimal periodic solution or washout state (Figure 2 of \cite{elgindy2025sustainable}) confirms the existence of stable attractors.
    \item Figures \ref{fig:NewFigNov1} and \ref{fig:NewFigNov2} demonstrate the current numerical optimization method's robustness to initial guess variations, with perturbed initial conditions converging to the same OPS, confirming the reliability of the FG-PS discretization combined with the edge-detection correction technique.
\end{itemize}

The small residuals in Figure \ref{fig:Fig10} further validate that the numerical solution accurately satisfies the system dynamics, indicating numerical stability of the computational approach.

\subsection{Practical Stability Implications for Water Treatment}
\label{subsec:dfvkdbk123}
The demonstrated stability of the periodic control scheme has important practical implications:
\begin{enumerate}[label=(\roman*)]
    \item The system maintains stable operation despite inflow variations and measurement noise.
    \item The bang-bang control strategy remains effective across different microbial populations (characterized by different $\alpha$ values).
    \item Stable periodic solutions ensure consistent pollutant removal efficiency over time.
\end{enumerate}

This stability analysis, combined with the local equilibrium stability results in \ref{app:LSE}, provides a comprehensive understanding of the system's dynamic behavior under the proposed fractional-order periodic control strategy, ensuring its suitability for real-world water treatment applications.

\section{Sensitivity Analysis and Numerical Simulations}
\label{sec:NTP1}
This section presents a detailed sensitivity analysis of the RFOCP to elucidate the influence of its key parameters on system performance. One of the primary objectives is to quantify how the fractional order $\alpha$, the dynamic scaling parameter $\vartheta$, and the memory length $L$ affect the OPC strategy and the resulting average substrate concentration $s\av$. Understanding these relationships is crucial for translating the theoretical fractional-order framework into practical, tunable control strategies for biological water treatment. To this end, we systematically vary one parameter at a time, holding others constant at their baseline values to isolate its effect on the optimal solution.

To support our findings in this work, consider the test case of the RFOCP with the key parameters summarized in Table \ref{tab:parameters}. Pseudospectral methods have been widely applied to solve FOCPs, offering adaptability to non-integer dynamics \cite{sahabi2024fractional,yang2024pseudospectral,kheirkhah2023legendre,ejlali2017pseudospectral,
habibli2019fractional,ali2019space}. The FG-PS method used here is particularly suited for periodic problems with sliding memory, as it uses Fourier expansions for global approximation while handling discontinuities via edge-detection corrections.

We shall use this test case as a benchmark for analyzing the influence of CFDS memory effects on the performance of fractional-order periodic control strategies in biological water treatment. This test problem is particularly challenging due to the bang-bang nature of the OC, which introduces discontinuities and requires specialized edge-detection techniques for accurate resolution.

All numerical simulations were carried out using \texttt{MATLAB R2023b}, installed on a personal laptop equipped with an \texttt{AMD Ryzen 7 4800H} processor (2.9 GHz, 8 cores/16 threads), 16 GB of RAM, and running \texttt{Windows 11}. The numerical optimization was performed over the full admissible control space $\mathcal{D}$. No a priori assumption was made about the bang-bang structure of the control. Nevertheless, the optimized solutions consistently exhibited bang-bang behavior in all simulations, in alignment with the theoretical results derived from the PMP analysis in Section \ref{sec:optimal_control}. This numerical observation further validates the Hamiltonian-based conclusion that singular arcs cannot exist for the RFOCP, and the optimal control must switch between its extremal values. All numerical simulations were performed assuming the periodic boundary condition $s(0) = s(T) = \bs$ holds. This constraint ensures that the substrate concentration (i.e. the pollutant level in wastewater treatment) at the start and end of each periodic cycle matches the steady-state concentration $\bs$. Biologically, it implies that the microbial environment resets to a baseline state where the substrate-dependent specific growth rate satisfies $\nu(\bs) = \bD$, thereby balancing microbial growth and washout. This setup allows periodic variations in the dilution rate $D$ to exploit dynamic microbial responses for improved performance, particularly in terms of average pollutant level reduction.

\begin{table}[ht]
\centering
\begin{tabular}{|l|c|l|}
\hline
\textbf{Parameter} & \textbf{Value} & \textbf{Description} \\
\hline
$\Sin$     & $8\ \text{mg/L}$       & Input substrate concentration \\
$D_{\min}$          & $0.02\ \text{h}^{-1}$  & Minimum dilution rate \\
$D_{\max}$          & $1.95\ \text{h}^{-1}$  & Maximum dilution rate \\
$\mu_{\max}$        & $2\ \text{h}^{-1}$     & Maximum growth rate \\
$K$                 & $5$                   & Saturation constant \\
$Y$                 & $1$                   & Yield coefficient \\
$\bD$           & $0.5\ \text{h}^{-1}$   & Average dilution rate \\
$T$                 & $15$                  & Control period \\
$\alpha$            & $0.85$                & Fractional order \\
$L$                 & $5$                   & Sliding memory length \\
$\vartheta$         & $0.25 \text{h}$          & Dynamic Scaling Parameter \\
\hline
\end{tabular}
\caption{Parameter values used in the numerical test problem.}
\label{tab:parameters}
\end{table}

The parameter values in Table \ref{tab:parameters} are selected based on the Contois chemostat model for wastewater treatment \cite{bayen2020improvement}, with adjustments for numerical demonstration and to satisfy theoretical requirements. For example, the given $K$ and $Y$ values in the table ensures $KY > 1$, which guarantees strict convexity of $\nu(s)$ in Theorems \ref{thm:nonconstant_existence}--\ref{thm:uniqueness}. The condition $\bD < \mu_{\max}$ ensures the existence of a non-washout equilibrium, as required for positive biomass persistence in Contois models \cite{elgindy2025sustainable}. Meanwhile, $\Sin = 8$ mg/L represents a feasible inlet concentration in lab-scale studies with synthetic wastewater \cite{grady2011biological}. The dilution rate bounds $D_{\min} = 0.02$ h$^{-1}$ and $D_{\max} = 1.95$ h$^{-1}$ span operational ranges feasible in laboratory bioreactors, with $D_{\max}$ chosen specifically to satisfy $\bD < \mu_{\max}$ while approaching the maximum growth rate. The Contois saturation constant $K = 5$ (dimensionless) introduces a significant biomass inhibition effect, making the specific growth rate $\mu(s, x)$ dependent on the biomass concentration $x$. This effectively models the crowding, diffusion limitations, and intensified competition for resources that occur in high-density microbial systems, providing a more realistic representation than non-inhibitory models like Monod kinetics. The period $T = 15$ h corresponds to typical hydraulic retention cycles in lab-scale bioreactors and permits observation of multiple control switches within one cycle under bang-bang strategies \cite{bayen2020improvement,grady2011biological}. The fractional order $\alpha = 0.85$ introduces moderate memory effects while maintaining numerical stability. The memory length $L = 5$ h is chosen to capture short-term microbial adaptation timescales, motivated by the established capability of fractional calculus to capture history-dependent processes and power-law decays in biological systems. The scaling parameter $\vartheta = 0.25$ h ensures dimensional homogeneity in the fractional equations (balancing the unit of the CFDS) while modulating the system's dynamic response amplitude \cite{elgindy2025sustainable}. These parameter choices guarantee the existence of positive, bounded solutions (see Corollary \ref{Cor:Oct14}) and enable the demonstration of up to 40\% improvement in pollutant removal efficiency compared to steady-state operation, as we show later in this section.

We solved the RFOCP using the FG-PS method developed by \citet{elgindy2024fouriera,elgindy2024fourierb} for discretization, followed by the application of MATLAB's \texttt{fmincon} solver to handle the resulting constrained NLP problem. The predicted optimal state and control values at a set of $N$ equally spaced collocation points were subsequently corrected by incorporating an advanced edge-detection technique to refine the OC profile, based on the methodologies presented in \citet{elgindy2023new,elgindy2024optimal}. Finally, the corrected data were interpolated at another set of $M$ equally spaced nodes within the interval $[0, T]$. Comprehensive technical details about the employed numerical methods are available in their original introductory papers \cite{elgindy2024fouriera,elgindy2024fourierb,elgindy2023new,elgindy2024optimal}, and a brief description of our numerical approach for solving the problem is provided in \ref{app:NSMfStR1}. The complete numerical solution process is systematically outlined in \ref{app:algorithm}, which details the step-by-step methodology from system discretization through optimization to final numerical solution construction.

\begin{figure}[ht]
    \centering
    \includegraphics[width=10cm]{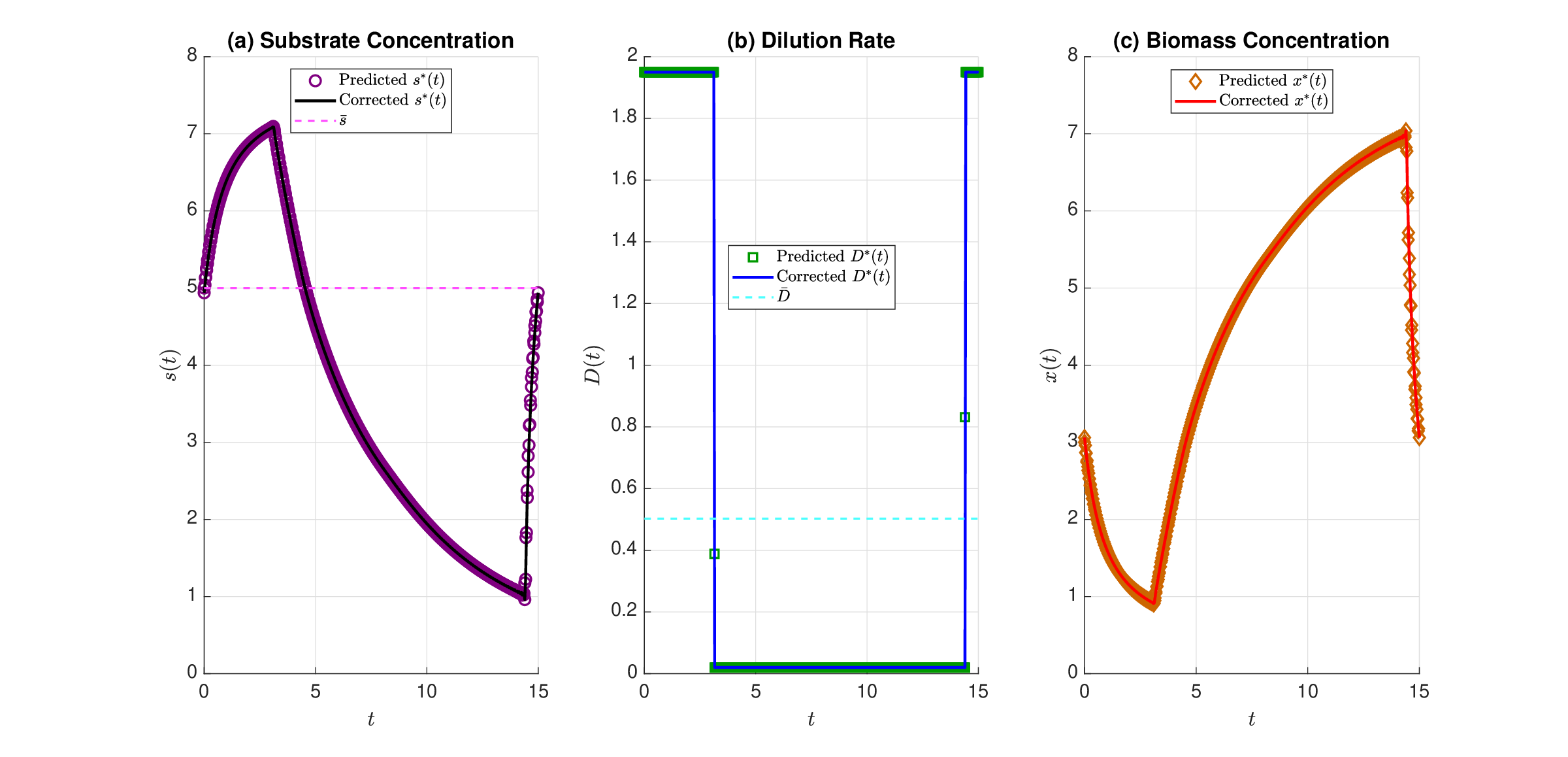}
    \caption{Time evolution (in hours) of (a) the PSC $s^*(t)$, (b) the OPC $D^*(t)$, and (c) the corresponding biomass concentration $x^*(t)$ of the RFOCP. The symbols  show the predicted solution values obtained at $N = 300$ equally-spaced collocation points from the numerical optimization, while the corrected solution (solid lines) is computed using a reconstructed bang-bang control law with $M=400$ interpolation points. Dashed lines indicate the average substrate concentration $\bs$ and average dilution rate $\bar{D}$, respectively.}
    \label{fig:Fig1}
\end{figure}

Figure \ref{fig:Fig1} illustrates the detailed time evolution of the optimal dilution rate $D^*(t)$ and the corresponding substrate and biomass concentrations, $s^*(t)$ and $x^*(t)$, respectively, over a full control period under the proposed fractional-order periodic strategy. At the onset of the cycle, $D^*(t)$ follows a bang-bang control pattern with abrupt switches between its extremal values occurring near $t = 3.131$ h and $t = 14.41$ h, rounded to four significant digits. This switching behavior induces strong fluctuations in $s^*(t)$ and $x^*(t)$. 

Initially, the high dilution rate rapidly introduces fresh substrate, causing $s^*(t)$ to rise. However, $x^*(t)$ decreases sharply because the specific growth rate under Contois kinetics, given by Eq. \eqref{eq:muContois1}, becomes temporarily too small to compensate for the elevated outflow rate. To elaborate further, despite $D_{\max} = 1.95\ \text{h}^{-1} < 2\ \text{h}^{-1} = \mu_{\max}$, the effective growth rate $\mu(s^*,x^*)$ depends on the biomass concentration. For instance, at $t = 0$, where $x^*(0) = 3\ \text{mg/L}$ and $s^*(0) = 5\ \text{mg/L}$, we find that $K x^* + s^* = 20$, yielding $\mu \approx 0.5\ \text{h}^{-1} \ll D_{\max}$. This mismatch causes the biomass to decline despite a theoretically sufficient maximum growth capacity.

As the control progresses, the dilution rate sharply decreases, limiting substrate inflow and enabling microbial consumption to reduce $s^*(t)$ to nearly $1\ \text{mg/L}$, indicating substantial substrate depletion. This phase supports efficient pollutant degradation while avoiding substrate overload. Toward the end of the cycle, the control switches back to $D_{\max}$, which helps reintroduce substrate and drives $x^*(t)$ down from its earlier peak of nearly $7\ \text{mg/L}$ to its initial value of $3\ \text{mg/L}$, thereby satisfying the periodic boundary condition. 

Importantly, the OPC strategy results in a lower average substrate concentration of $s\av \approx 3.622\ \text{mg/L}$, compared to the steady-state value $\bs = 5\ \text{mg/L}$, achieving a $27.56\%$ improvement in pollutant removal efficiency. The incorporation of memory effects through fractional-order dynamics improves the system responsiveness by accounting for past states in the evolution of substrate and biomass concentrations. This nonlocal behavior leads to more robust control outcomes, improving stability and performance over time.

Figure \ref{fig:Fig2} shows the trajectories of the approximate OPSs obtained at $N = 400$ and $M = 500$. The plots appear visually indistinguishable from Fig.~\ref{fig:Fig1}, which were generated at $N = 300$ and $M = 400$. This strong agreement between solutions at different resolutions indicates that the numerical method has converged and is accurately resolving the system dynamics, including the sharp switching behavior of the bang-bang control.

\begin{figure}[ht]
    \centering
    \includegraphics[width=10cm]{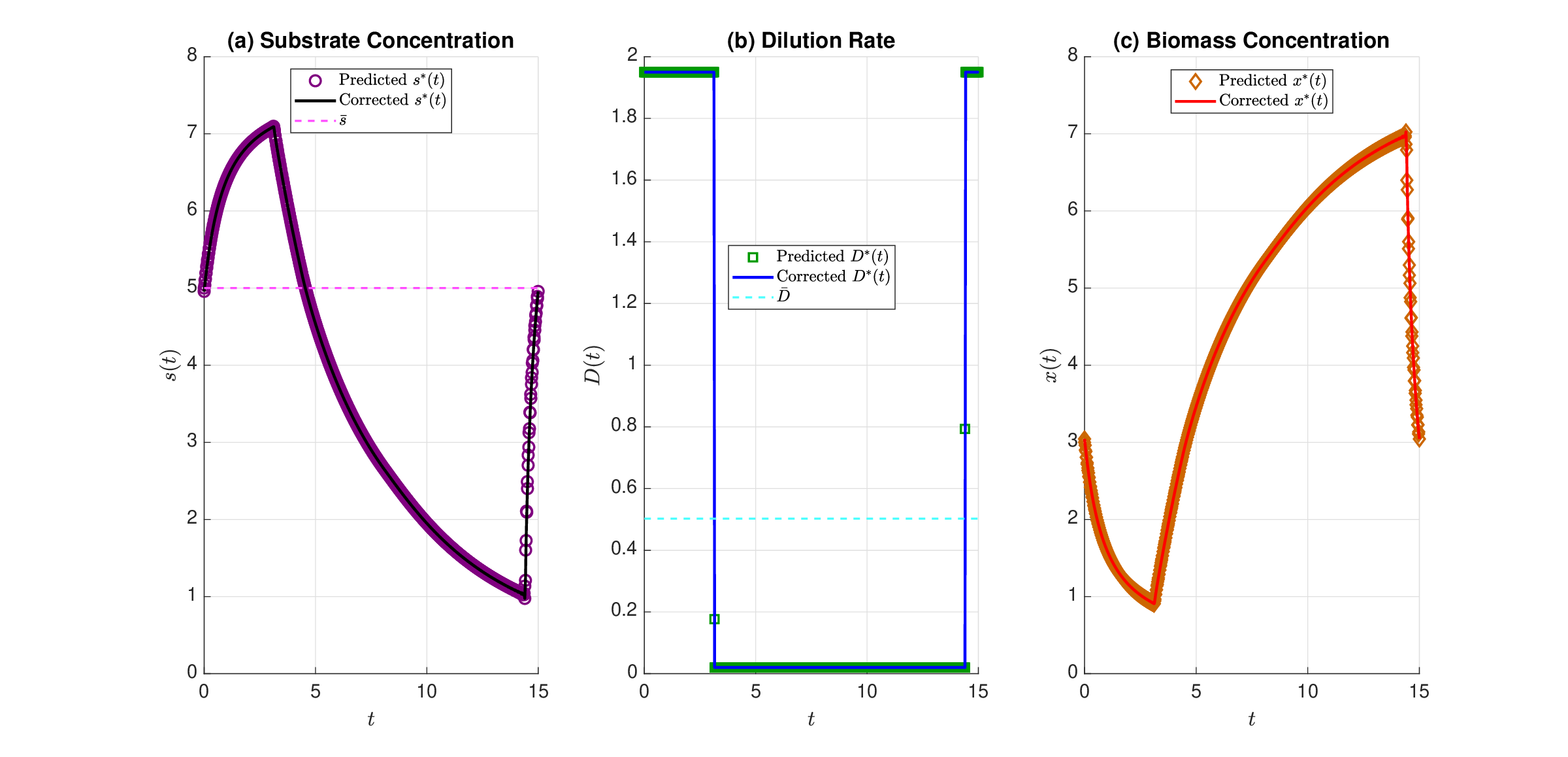}
    \caption{Time evolution (in hours) of (a) the PSC $s^*(t)$, (b) the OPC $D^*(t)$, and (c) the biomass concentration $x^*(t)$ of the RFOCP. The symbols  show the predicted solution values obtained at $N = 400$ equally-spaced collocation points from the numerical optimization, while the corrected solution (solid lines) is computed using a reconstructed bang-bang control law with $M=500$ interpolation points. Dashed lines indicate the average substrate concentration $\bs$ and average dilution rate $\bar{D}$, respectively.}
    \label{fig:Fig2}
\end{figure}

\begin{figure}[ht]
    \centering
    \includegraphics[scale=0.4]{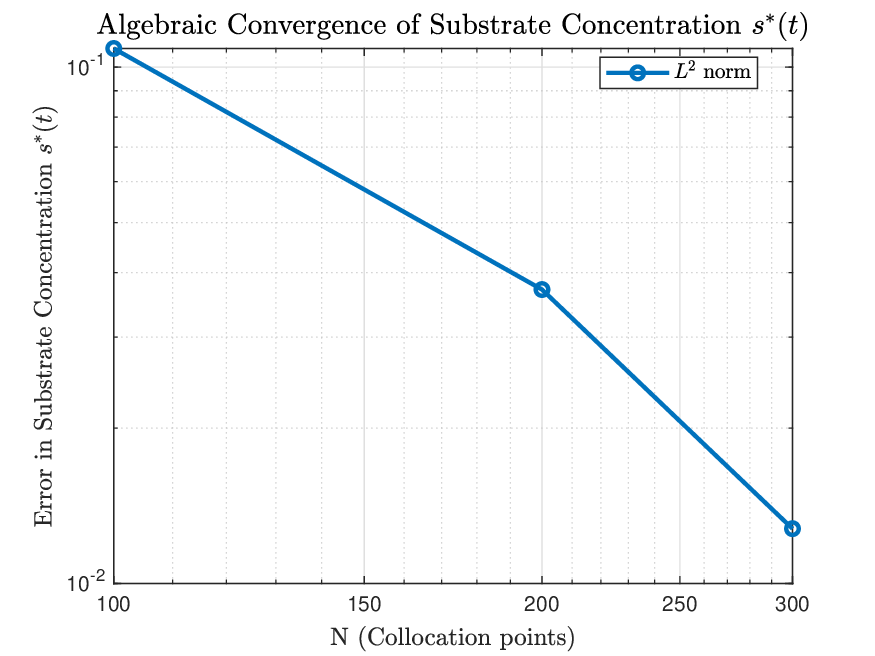}
    \caption{Algebraic convergence of the PSC $s^*(t)$ for the RFOCP. The plot shows the $L^2$-error norm in $s^*(t)$ as a function of the number of collocation points $N$. The reference solution is computed at $N=400$.}
    \label{fig:Fig3}
\end{figure}

\begin{figure}[ht]
    \centering
    \includegraphics[scale=0.4]{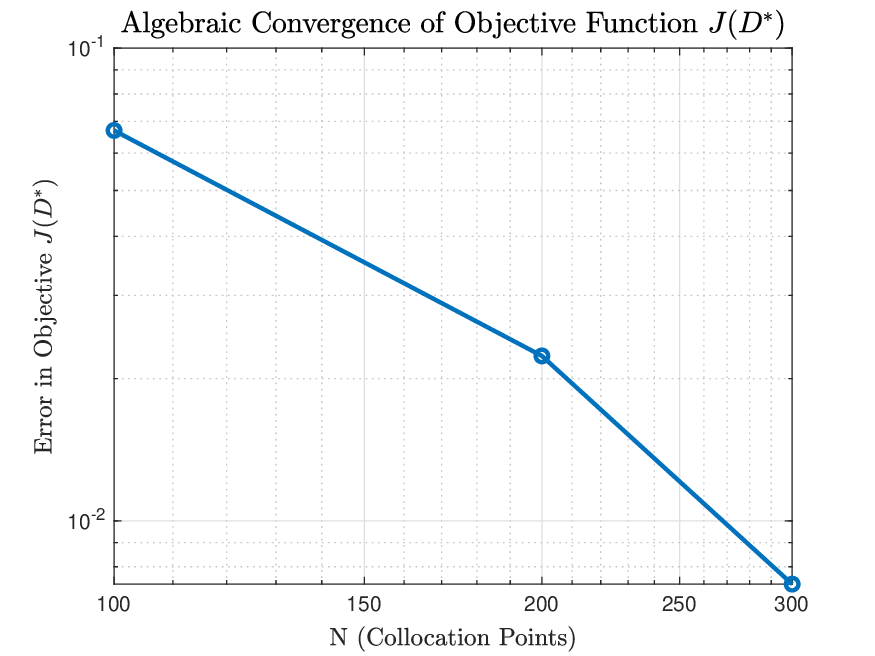}
    \caption{Algebraic convergence of the OOFV $J(D^*)$ for the RFOCP. The absolute error in the computed objective value is shown as a function of the number of collocation points $N$, with the reference value taken at $N=400$.}
    \label{fig:Fig4}
\end{figure}

To further validate the numerical convergence of the FG-PS method, we solved the RFOCP for several values of the collocation parameter $N \in \{100, 200, 300, 400\}$. The primary objective of this analysis was to examine the convergence behavior of the PSC $s^*$, the OPC $D^*$, and the corresponding OOFV $J(D^*)$. For each value of $N$, the corrected numerical solutions were interpolated onto a common finer grid of $M = 500$ equispaced points to facilitate consistent comparison against a reference solution computed using $N = 400$. Figures \ref{fig:Fig3} and \ref{fig:Fig4} illustrate the convergence and accuracy characteristics of the method. Specifically, Figure \ref{fig:Fig3} demonstrates the algebraic convergence of the PSC $s^*(t)$, as reflected by the decay in the $L^2$-error norm with increasing $N$. Figure \ref{fig:Fig4} presents the convergence behavior of the OOFV $J(D^*)$, with the absolute error steadily decreasing as $N$ increases. Remarkably, the switching times agree to full machine precision at about $t = 3.131$ h and $t = 14.41$ h, rounded to four significant digits, across all discretization levels ($N = 100, 200, 300, 400$), demonstrating perfect numerical reproduction of the control structure's temporal features, despite its discontinuous, bang-bang nature. These results confirm that the FG-PS method, equipped with the edge-detection correction technique, produces robust and accurate approximations of the state and control variables, as well as the associated performance index, even in the presence of nonsmooth control profiles. Furthermore, Figure \ref{fig:Fig3} demonstrates an algebraic convergence decay in the $L^2$-norm of the errors in $s^*$ with respect to $N$. This behavior aligns with the expected reduction in global spectral convergence rates due to the discontinuities inherent in the bang-bang control $D^*$. The consistent reduction in the absolute error of the average substrate concentration $s^*\av$ with increasing $N$ further supports the reliability of the method in resolving the system dynamics and the sharp switching behavior of the OC. 

Figures \ref{fig:Fig5} and \ref{fig:Fig6} offer valuable insights into the behavior of the OC structure and its corresponding performance as the fractional order $\alpha$ varies. Figure \ref{fig:Fig5} illustrates the switching times $\xi_k$ of the optimal bang-bang control across different values of $\alpha$ in the FOCM. Each plotted symbol represents a distinct switching event, revealing how the number and location of switching points are sensitive to the memory effect introduced by the fractional-order dynamics. A detailed summary of the number and approximate locations of switches, along with the corresponding average substrate concentrations $s\av^*$, is provided in Table \ref{tab:alpha_switches}. The results in this table highlight that while the number of switches remains even, as guaranteed by the PMP analysis, their frequency and positions vary nonlinearly with $\alpha$, reflecting the nonlocal influence of historical states. A notable trend is observed here where lower values of the fractional order parameter $\alpha$ result in a higher number of control switches in the bang-bang control strategy. This increased switching frequency at lower $\alpha$ can be attributed to the stronger memory effect of the CFDS, which necessitates more frequent adjustments in the dilution rate $D^*$ to maintain optimal substrate concentration. In other words, to counterbalance the inertia introduced by strong memory at low $\alpha$, the OPC must respond more frequently, resulting in a higher number of switches to steer the system effectively within the constraints. Complementarily, Figure \ref{fig:Fig6} depicts the average substrate concentration $s^*\av$ achieved under the OC for varying $\alpha$. Interestingly, $s^*\av$ increases from $\alpha = 0.1$ to $\alpha = 0.3$, peaks at $\alpha = 0.3$, and then shows a monotonic decrease as $\alpha$ increases to $0.4$, $0.5$, $0.6$, $0.7$, $0.8$, and $0.9$. This observation indicates that the performance does not decrease uniformly with increasing $\alpha$, and intermediate values such as $\alpha = 0.3$ may yield higher average substrate concentrations than expected. Moreover, the results manifest that for all tested fractional orders, there exist non-constant periodic control strategies that outperform the corresponding steady-state solutions, yielding lower average substrate concentrations and improving pollutant removal efficiency. These findings confirm that tuning the fractional order serves as a powerful lever for improving system performance and that the effectiveness of periodic control relative to steady-state operation depends critically on the degree of memory in the system. Consequently, selecting microbial species with inherently low memory effects (i.e., high fractional order $\alpha$ close to $1$) can significantly improve water quality, as such species respond more effectively to time-varying OC strategies.

\begin{figure}[ht]
    \centering
    \includegraphics[scale=0.4]{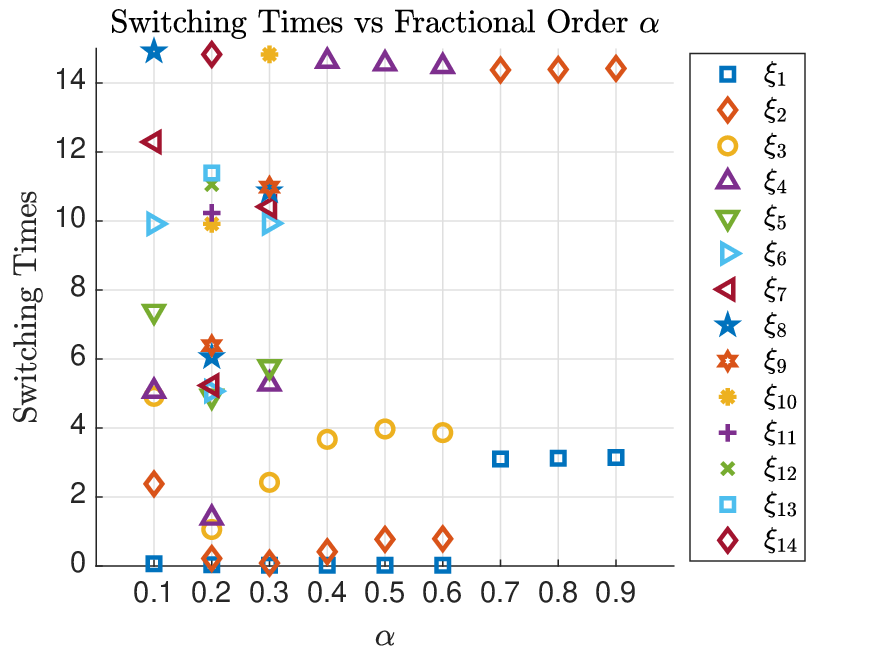}
    \caption{Switching times $\xi_k$ (in hours) of the optimal bang-bang control as a function of the fractional order $\alpha$ in the FOCM, obtained using $N = 300$ and $M = 400$. All other parameter values were taken from Table \ref{tab:parameters}. Each symbol corresponds to a different switching event, illustrating how the control structure changes with the order of the FD.}
    \label{fig:Fig5}
\end{figure}

\begin{table}[ht]
\centering
\resizebox{0.45\textwidth}{!}{%
\begin{tabular}{|c|c|>{\raggedright\arraybackslash}p{9cm}|c|}
\hline
$\alpha$ & \makecell{No. of\\Switches} & Approximate Switching Times ($\xi_k$) & $s\av^*$ \\
\hline
0.1 & 8 & 0.0676, 2.380, 4.917, 5.068, 7.380, 9.917, 12.29, 14.92 & 3.839 \\
\hline
0.2 & 14 & \makecell[l]{0.0375, 0.2177, 1.059, 1.389, 4.917, 5.068, 5.233, 6.059, 6.389,\\ 9.917, 10.23, 11.06, 11.39, 14.83}  & 3.980 \\
\hline
0.3 & 10 & \makecell[l]{0.0225, 0.0826, 2.425, 5.278, 5.773, 9.932, 10.41, 10.86, 10.98,\\ 14.83} & 4.280 \\
\hline
0.4 & 4 & 0.0225, 0.4129, 3.671, 14.63 & 4.237 \\
\hline
0.5 & 4 & 0.0225, 0.7733, 3.971, 14.56 & 4.104 \\
\hline
0.6 & 4 & 0.0225, 0.7883, 3.866, 14.48 & 4.023 \\
\hline
0.7 & 2 & 3.101, 14.38 & 3.798 \\
\hline
0.8 & 2 & 3.123, 14.39 & 3.687 \\
\hline
0.9 & 2 & 3.146, 14.42 & 3.596 \\
\hline
\end{tabular}
}
\caption{Number and approximate locations of control switches for different values of $\alpha$.}
\label{tab:alpha_switches}
\end{table}

\begin{figure}[ht]
    \centering
    \includegraphics[scale=0.4]{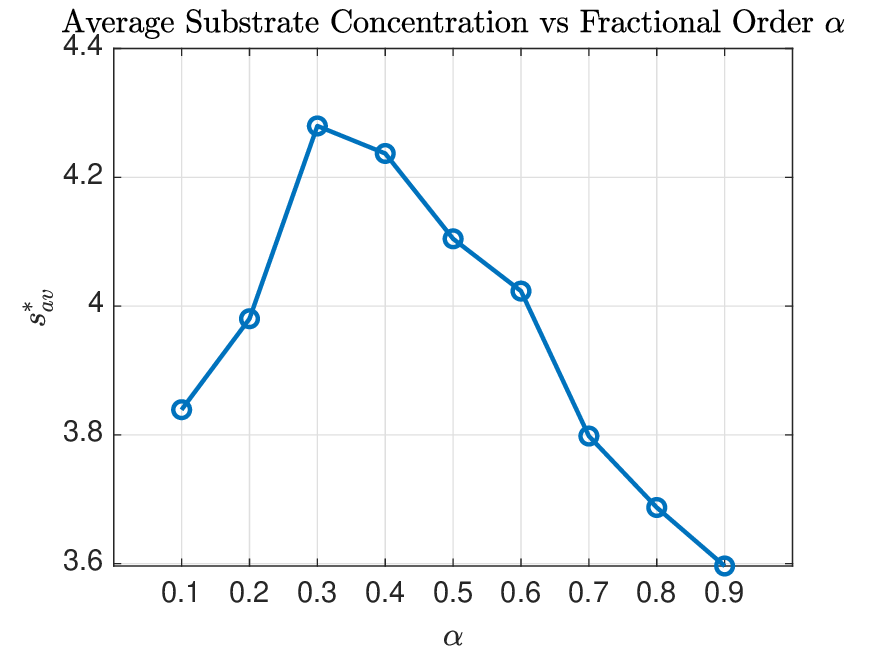}
    \caption{Sensitivity to Fractional Order $\alpha$: Average substrate concentration $s^*\av$ as a function of the fractional order $\alpha$ for the OC of the FOCM, obtained using $N = 300$ and $M = 400$. The plot is generated for $\alpha \in \{0.1, 0.2, 0.3, 0.4, 0.5, 0.6, 0.7, 0.8, 0.9\}$. All other parameter values were taken from Table \ref{tab:parameters}. The results are obtained by solving the RFOCP for various values of $\alpha$ using the specified system and optimization parameters.}
    \label{fig:Fig6}
\end{figure}

In Figure \ref{fig:Fig7}, the effect of the sliding memory length $L$ on the average substrate concentration $s\av$ is analyzed in the context of the RFOCP, where $L$ exclusively influences the CFDS. The figure shows that as $L$ increases from 0.5 to 1.5, $s\av$ slightly increases, suggesting a mild degradation in performance when the memory window is too short to capture sufficient historical dynamics. Beyond $L = 1.5$, $s\av$ declines consistently with increasing $L$, indicating improved pollutant removal efficiency as the CFDS incorporates a richer history of the system's state evolution. This trend continues until approximately $L = 10$, after which the curve flattens, implying that the marginal benefit of extending the memory window diminishes. In other words, beyond $L = 10$, the benefits plateau, suggesting a point of diminishing returns where extending the memory window no longer yields significant performance gains. Since $L$ directly affects the memory range of the FD, this behavior highlights the importance of tuning $L$ to balance the cost and accuracy of the approximate FD with the benefits of nonlocal memory effects. Under the given data, moderate values of $L$ (like around $10$) are sufficient to exploit the memory structure effectively for OC performance.

\begin{figure}[ht]
    \centering
    \includegraphics[scale=0.4]{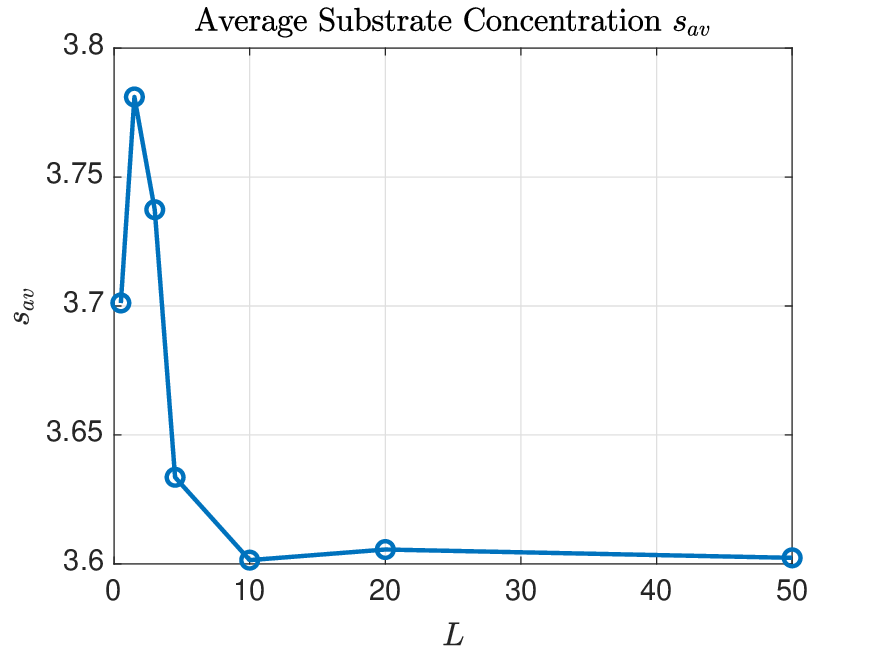}
    \caption{Sensitivity to Memory Length $L$: Dependence of the average substrate concentration $s\av$ on the sliding memory length $L$ for the RFOCP. The plot is generated for $L \in \{0.5,\, 1.5,\, 3,\, 4.5,\, 10,\, 20,\, 50\}$. All other parameter values were taken from Table \ref{tab:parameters}.}
    \label{fig:Fig7}
\end{figure}

Figure \ref{fig:Fig8} illustrates the impact of the dynamic scaling parameter $\vartheta$ on the average substrate concentration $s\av$ for biological water treatment. The plot shows a monotonic decrease in $s\av$ as $\vartheta$ increases, reflecting the scaling effect of $\vartheta^{1-\alpha}$ on the right-hand side of the FDE \eqref{eq:reducedFDE} governing the chemostat dynamics. Larger $\vartheta$ values amplify the magnitude of the system dynamics, which improves the responsiveness of microbial activity to control inputs, leading to more effective pollutant degradation and lower $s\av$. Conversely, smaller $\vartheta$ values reduce the dynamic response, resulting in higher $s\av$ due to less effective substrate consumption. This trend highlights the importance of tuning $\vartheta$ to optimize the system's dynamic response, complementing the role of the fractional order $\alpha$, where higher $\alpha$ (weaker memory effects) further improves performance by reducing the influence of historical states, as shown in Figure \ref{fig:Fig6}. Notice that the reduction in the minimum average substrate concentration at $\vartheta = 32$, where $s\av \approx 3.001 \, \text{mg/L}$ compared to steady-state operation $\bs = 5 \, \text{mg/L}$, is approximately $39.98\%$. This nearly 40\% reduction is substantial and serves as persuasive evidence that fractional-order control with properly tuned parameters (here $\vartheta = 32$) can significantly outperform steady-state strategies.

\begin{figure}[ht]
    \centering
    \includegraphics[scale=0.4]{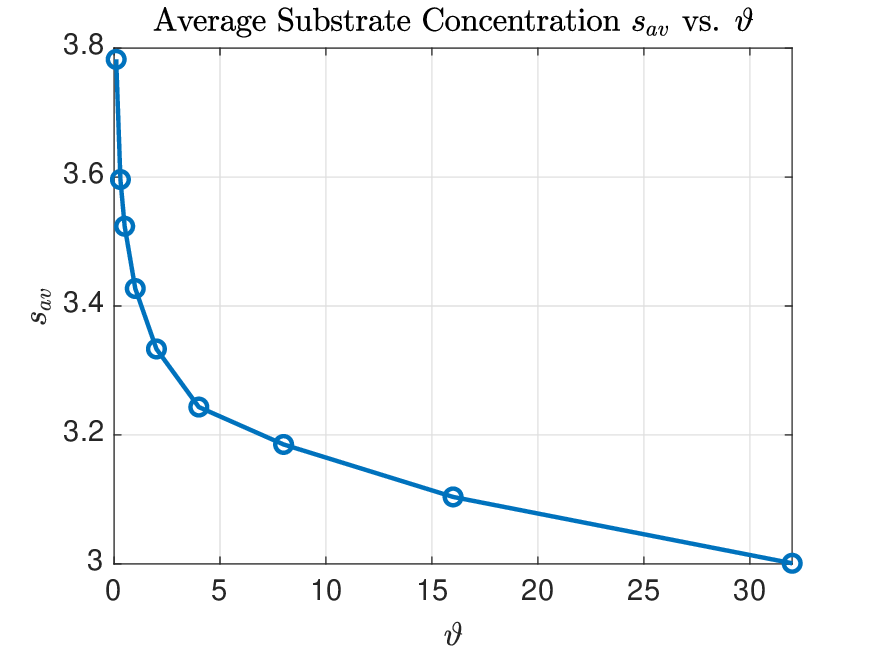}
    \caption{Sensitivity to Scaling Parameter $\vartheta$: Average substrate concentration $s\av$ as a function of the parameter $\vartheta$. The plot is generated for $\vartheta \in \{0.1, 0.3, 0.5, 1, 2, 4, 8, 16, 32\}$. All other parameter values were sourced from Table \ref{tab:parameters}.}
    \label{fig:Fig8}
\end{figure}

Figure \ref{fig:Fig9} shows how $s\av$ varies with $T$ in the range from 1 to 20 hours. We clearly see that $s\av$ decreases as the control period $T$ increases, reflecting improved pollutant removal efficiency in the bioprocess. This indicates that longer periodic cycles provide microorganisms sufficient time to adapt to changing environmental conditions and dilution regimes, thus improving substrate uptake. In contrast, shorter $T$ values may not permit adequate synchronization between the dilution rate and the slower microbial growth responses governed by Contois kinetics, resulting in suboptimal pollutant degradation. Therefore, tuning $T$ appropriately improves system responsiveness and biological efficiency, underscoring the importance of harmonizing periodic control inputs with the intrinsic adaptation timescales of microbial populations.

\begin{figure}[ht]
    \centering
    \includegraphics[scale=0.4]{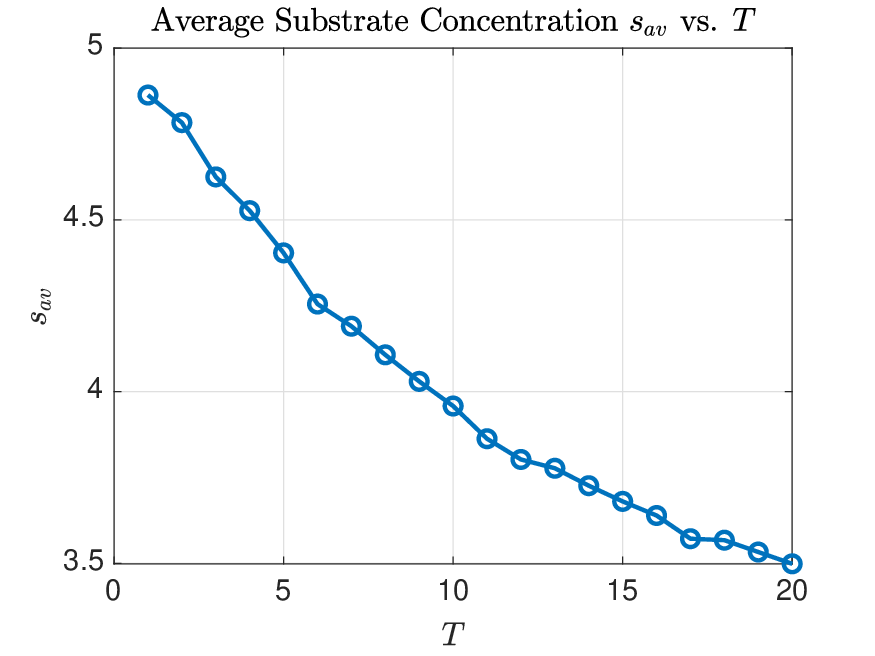}
    \caption{Sensitivity to Time Horizon $T$ (in hours): Average substrate concentration $s\av$ as a function of the time horizon $T$ for $L = 4$. All other parameter values were sourced from Table \ref{tab:parameters}.}
    \label{fig:Fig9}
\end{figure}

Figure \ref{fig:Fig10} displays the residuals associated with the FDEs governing the substrate and biomass concentrations in the 2D FOCS. The figure serves to validate the analytical expression \eqref{eq:biomass2} by independently solving the 2D FOCS. The consistently small residual values across the entire time interval confirm the high accuracy of the numerical approximations and support the validity of the substrate and biomass dynamics under the optimal dilution control strategy.

To further validate the robustness of our numerical approach, Figures \ref{fig:NewFigNov1} and \ref{fig:NewFigNov2} present the OPSs obtained when the initial guesses for the state and control variables are perturbed by $0.1$ and $0.2$, respectively. Remarkably, both simulations converge to the same optimal solution, achieving an average substrate concentration of $s\av \approx 3.622$ mg/L. The resulting PSC, OPC, and biomass concentration profiles are visually identical to those shown earlier in Figure \ref{fig:Fig1}, demonstrating the numerical method's insensitivity to initial guess variations. This consistency confirms the reliability of the FG-PS discretization combined with the edge-detection correction technique. The method's robustness is particularly valuable for practical applications where precise initial conditions may not be known a priori, ensuring that the OPC strategy can be reliably computed for real-world bioprocess optimization.

\begin{figure}[ht]
    \centering
    \includegraphics[scale=0.4]{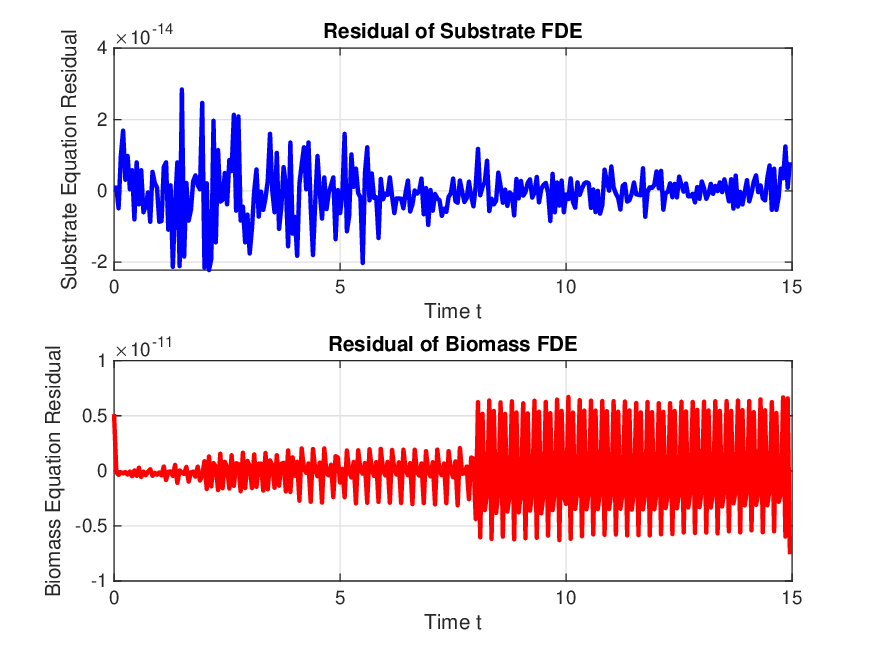}
    \caption{Residuals of the fractional chemostat model equations over time (in hours). \textit{Top:} The residuals of the substrate concentration FDE, representing the difference between the computed CFDS and the model's right-hand side at collocation points. \textit{Bottom:} The residuals of the biomass concentration FDE, computed similarly.}
    \label{fig:Fig10}
\end{figure}

\begin{figure}[ht]
    \centering
    \includegraphics[width=10cm]{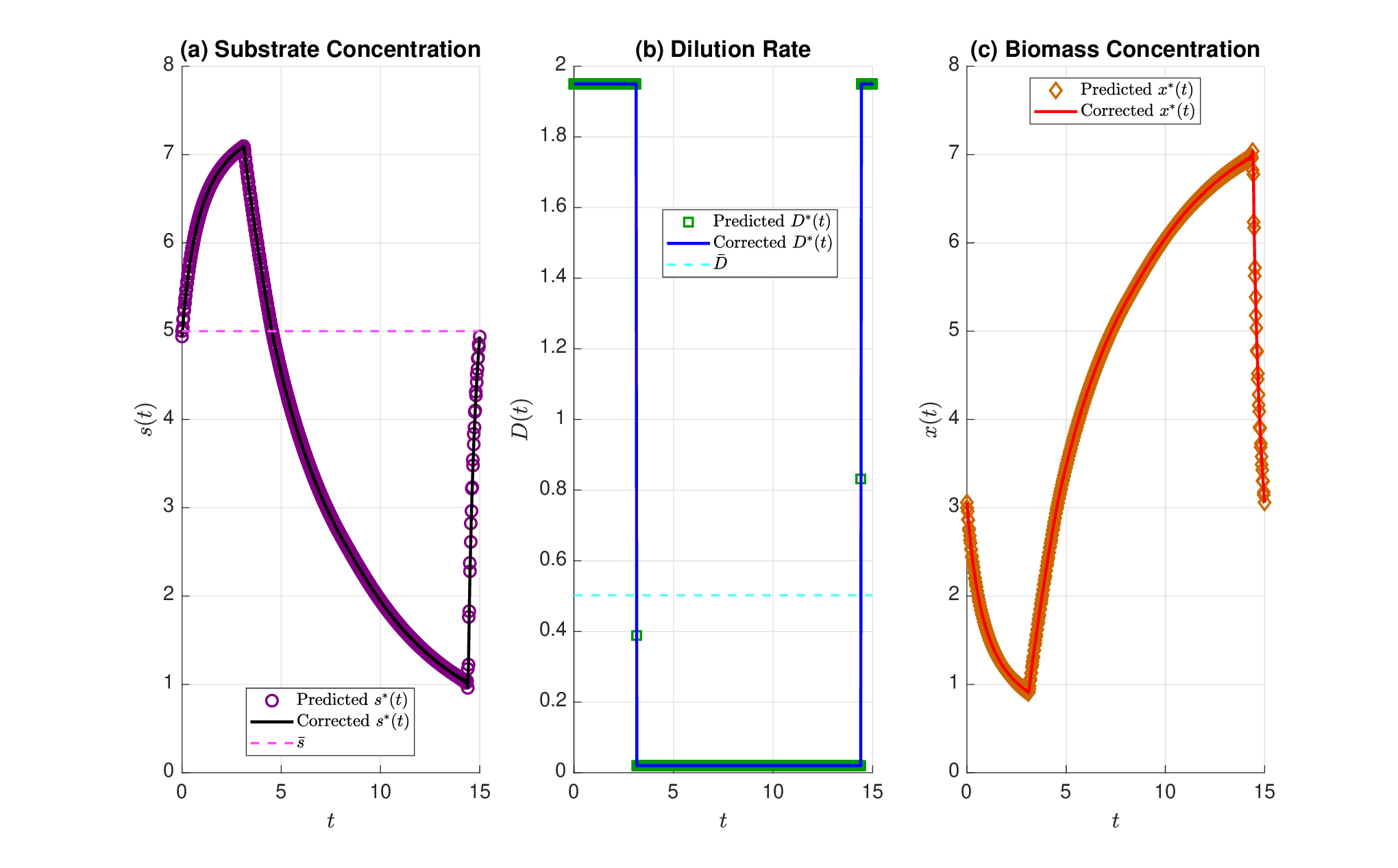}
    \caption{Time evolution (in hours) of (a) the PSC $s^*(t)$, (b) the OPC $D^*(t)$, and (c) the corresponding biomass concentration $x^*(t)$ of the RFOCP. The symbols  show the predicted solution values obtained at $N = 300$ equally-spaced collocation points from the numerical optimization, where the initial solution value guesses have been perturbed by $0.1$. The corrected solution (solid lines) is computed using a reconstructed bang-bang control law with $M=400$ interpolation points. Dashed lines indicate the average substrate concentration $\bs$ and average dilution rate $\bar{D}$, respectively.}
    \label{fig:NewFigNov1}
\end{figure}

\begin{figure}[ht]
    \centering
    \includegraphics[width=10cm]{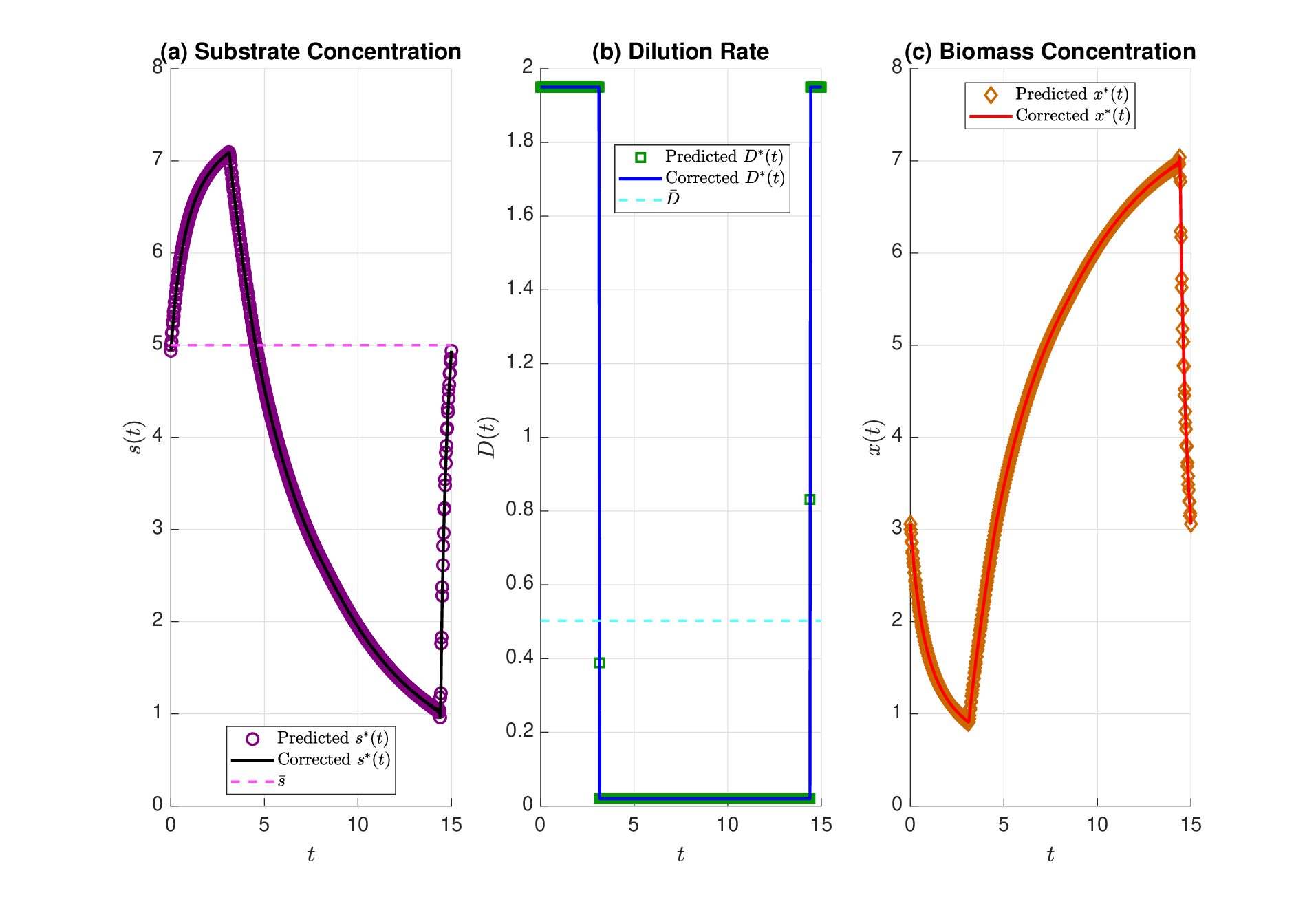}
    \caption{Time evolution (in hours) of (a) the PSC $s^*(t)$, (b) the OPC $D^*(t)$, and (c) the corresponding biomass concentration $x^*(t)$ of the RFOCP. The symbols  show the predicted solution values obtained at $N = 300$ equally-spaced collocation points from the numerical optimization, where the initial solution value guesses have been perturbed by $0.2$. The corrected solution (solid lines) is computed using a reconstructed bang-bang control law with $M=400$ interpolation points. Dashed lines indicate the average substrate concentration $\bs$ and average dilution rate $\bar{D}$, respectively.}
    \label{fig:NewFigNov2}
\end{figure}

\section{Conclusion and Discussion}
\label{eq:Conc}
This study has delved into the intricate dynamics of bioprocesses, specifically focusing on the role of fractional-order calculus in modeling microbial memory and its implications for OC strategies in chemostat systems. Our findings underscore the critical importance of the fractional-order parameter $\alpha$, which serves as a quantitative metric for microbial adaptation latency and memory effects. A lower $\alpha$ signifies a microbial population with sluggish adaptive responses, characterized by a strong `long-term memory' of past nutrient levels and environmental conditions. This biological inertia necessitates highly dynamic and frequent adjustments of the optimal dilution rate to maintain effective pollutant degradation and biomass stability, as evidenced by our numerical simulations in Section \ref{sec:NTP1}. Conversely, a higher $\alpha$ indicates a more agile microbial community that adapts rapidly to environmental fluctuations, allowing the OPC scheme to achieve efficient substrate removal with fewer interventions. This observed inverse relationship between $\alpha$ and the control switching frequency highlights a fundamental trade-off: a high switching frequency at low $\alpha$ acts as a compensatory mechanism for inherent biological sluggishness, while fewer switches at high $\alpha$ reflect a system that is intrinsically more responsive and requires minimal external intervention to sustain performance. These insights are crucial for designing robust and efficient bioprocesses, suggesting that systems populated by slow-adapting microorganisms demand more complex and energy-intensive control strategies. Conversely, the selection or engineering of faster-adapting strains (i.e., those exhibiting higher $\alpha$ values) could significantly simplify control demands, leading to more cost-effective and streamlined system designs. The profound connection between the fractional-order $\alpha$ and control switching frequency strongly emphasizes the necessity of integrating memory effects into both the theoretical modeling and practical optimization of real-world bioprocesses.

Beyond these fundamental insights, this work has made several significant contributions to the field of bioprocess engineering and fractional calculus: (i) We developed a pioneering FOCM that incorporates the CFDS effect. This model extends the traditional integer-order framework, offering a more realistic representation of microbial growth and substrate degradation by capturing memory-dependent dynamics and non-local effects that are prevalent in complex biological systems. This advancement provides a more accurate predictive tool for bioprocess behavior. (ii) We successfully reduced the 2D FOCS to a 1D FDE through a novel transformation linking substrate and biomass concentrations. This simplification, derived from established principles, significantly improves analytical and computational efficiency without compromising the essential dynamics of the system, thereby facilitating more tractable optimization problems. (iii) We rigorously formulated an FOCP aimed at minimizing the average pollutant concentration under OPC, explicitly integrating fractional-order dynamics and practical constraints on dilution rates and treatment capacity. This formulation extends the applicability of periodic control strategies to fractional-order systems, addressing a critical gap in existing literature. (iv) We established the existence of OPSs for the RFOCP. Furthermore, we provided conditions for the uniqueness of these solutions under specific parameter regimes, thereby improving the theoretical robustness and practical reliability of the proposed control approach. (v) By optimizing OPC within a fractional-order context, this research offers tangible insights into improving the efficiency of water treatment processes. The model's ability to account for historical system behavior through its memory effect can guide the design of more effective control strategies, potentially leading to reduced operational costs and a diminished environmental footprint, contributing directly to cleaner water and healthier ecosystems.

Complementing the role of the fractional order $\alpha$, our results reveal that the dynamic scaling parameter $\vartheta$ also plays a pivotal role in optimizing system behavior. Specifically, increasing $\vartheta$ intensifies the system's sensitivity to control inputs, accelerating pollutant degradation and enabling sharper, more effective bang-bang control responses. Numerical simulations confirm that properly tuned $\vartheta$ values can lead to reductions in average substrate concentration of up to 40\%, far outperforming steady-state control strategies. Thus, $\vartheta$ serves as a critical design parameter in boosting the effectiveness of fractional-order control schemes.

Complementing these findings, our results indicate that the sliding memory length $L$ plays an important role in utilizing historical system behavior to improve bioprocess performance. As $L$ increases, the CFDS incorporates a broader temporal window, capturing persistent microbial dynamics and adaptation latency with greater accuracy. Numerical experiments show that moderate to large values of $L$ consistently yield lower average substrate concentrations, improving pollutant removal efficiency. Careful tuning of $L$ allows engineers to balance memory-driven responsiveness with practical implementation constraints, making it a key design parameter in fractional-order control systems.

Compared to existing integer-order methods (e.g., \cite{bayen2020improvement}), our fractional-order approach captures memory effects, leading to up to 40\% better pollutant removal by accounting for historical states in microbial dynamics. Relative to other traditional fractional control schemes, the CFDS preserves periodicity, reduces computational burden (finite window vs. full history), and enables efficient PS discretization for periodic problems, improving accuracy and scalability for real-time water treatment applications.

While this study focuses on fractional-order deterministic modeling, it is valuable to contextualize our approach relative to other modern modeling frameworks: (i) \textit{Stochastic Models}: Unlike stochastic approaches that incorporate random fluctuations in microbial growth or environmental conditions \cite{smeets2010stochastic}, our fractional-order deterministic framework captures systematic memory effects through the CFDS. Stochastic models excel at quantifying uncertainty but may require extensive statistical data, whereas our approach provides a deterministic characterization of memory-driven dynamics that complements stochastic analysis. (ii) \textit{Hybrid Systems}: Hybrid models combining continuous dynamics with discrete events (see \cite{teel2012hybrid}) may capture operational switches in water treatment plants. Our fractional-order framework offers an alternative by modeling continuous memory effects without requiring explicit discrete state transitions, providing a different perspective on the system dynamics. (iii) \textit{Machine Learning Approaches}: Data-driven methods like neural networks or reinforcement learning may capture complex patterns from operational data \cite{lowe2022review} but may lack interpretable physical and biological insights. Our model maintains biological interpretability through established microbial growth kinetics while sustaining them with fractional calculus, bridging first-principles modeling with data-informed memory effects. The fractional-order approach presented here particularly excels in capturing hereditary effects and long-range temporal dependencies, which are challenging to represent in conventional integer-order models and may require complex memory architectures in machine learning approaches.

In conclusion, this work represents a significant step towards a more nuanced understanding and effective control of bioprocesses by integrating the concept of microbial memory through fractional-order calculus. The theoretical advancements and practical implications presented herein lay a robust foundation for future research aimed at developing more sustainable and efficient biotechnological solutions for environmental and industrial challenges. Future work could explore time-varying upper bounds on state and control functions to model dynamic constraints in real-world bioprocesses such as varying inlet concentrations or seasonal environmental factors, as well as extensions to multi-compartment systems and more general fractional optimal control frameworks. Additionally, chaos analysis could provide valuable insights into the nonlinear dynamics of FOCSs under extreme parameter regimes or external disturbances, potentially revealing bifurcation behaviors and chaotic regimes that may inform operational boundaries for practical implementations. Furthermore, one may explore time-varying fractional orders $\alpha(t)$ to model adaptive microbial memory effects, where the fractional order varies based on environmental conditions or operational phases.  

\section*{Limitations and Future Validation}
While the proposed FOCM demonstrates significant theoretical and numerical advancements in modeling memory-driven bioprocesses, we acknowledge the importance of experimental validation. The current study focuses on developing a novel fractional-order theoretical framework and demonstrating its potential through high-fidelity simulations. The parameters are selected from literature-reported realistic ranges \cite{bayen2020improvement,grady2011biological} and adjusted for numerical and theoretical purposes. We hope that a future companion study could validate the FOCM against pilot-scale experimental data from a real wastewater treatment facility, where inlet/outlet substrate concentrations, biomass density, and dilution rate profiles will be measured and compared with model predictions. This validation would provide critical empirical support for the FOCM incorporating CFDS, confirming its ability to capture long-term memory and non-local effects in microbial population dynamics. By aligning model predictions with real-world data on periodic dilution, substrate reduction, and biomass concentration, the study would strengthen the practical applicability of the proposed periodic fractional control strategies for optimizing clean water production and ecosystem health preservation in wastewater treatment processes.

\section*{Declarations}

\subsection*{Competing Interests}
The authors declare that they have no competing interests.

\subsection*{Availability of Supporting Data}
All data generated or analyzed during this study are included in this published article.

\subsection*{Ethical Approval and Consent}
This study does not involve human participants, animal subjects, or clinical data, and therefore does not require ethical approval or consent.

\subsection*{Funding}
This article was supported by Ajman University Internal Research Grant No. 2025-IRG-CHS-1. The research findings presented in this article are solely the authors responsibility.

\section*{CRediT Author Statement}
Kareem T. Elgindy: Conceptualization, Methodology, Formal Analysis, Software, Investigation, Visualization, Writing Original Draft, Project Administration.\\
Muneerah Al Nuwairan: Formal Analysis, Review \& Editing.\\
Liew Siaw Ching: Formal Analysis, Review \& Editing.

\subsection*{Acknowledgment}

\appendix
\section{Lemma on the Integral of the CFDS}
\label{app:LIC1}
\begin{lemma}\label{lem:1}
Let $s \in AC_T$. Then, for any period $T > 0$, memory length $L > 0$, and fractional order $\alpha \in (0, 1]$, the integral of the CFDS over one period vanishes:
\begin{equation}\label{eq:CFDSzer1}
\int_0^T {}_L^{MC} D_t^\alpha s(t) \, dt = 0.
\end{equation}
\end{lemma}
\begin{proof}
We consider two cases:\\[-1em]

\noindent \textbf{Case 1: $\alpha \in (0, 1)$:} Since $s$ is $T$-periodic and absolutely continuous, its derivative $s'$ exists a.e., is Lebesgue integrable, and is also $T$-periodic. The CFDS is given by:
\[
{}_L^{MC} D_t^\alpha s(t) = \frac{1}{\Gamma(1 - \alpha)} \int_{t - L}^t (t - \tau)^{-\alpha} s'(\tau) \, d\tau.
\]
Substitute $u = t - \tau$:
\[
\int_{t - L}^t (t - \tau)^{-\alpha} s'(\tau) \, d\tau = \int_0^L u^{-\alpha} s'(t - u) \, du.
\]
If $t - u < 0$, we exploit the periodicity of $s'$: for any $u \in [0, L]$, there exists an integer $k$ such that $t - u + kT \in [0, T]$, and thus $s'(t - u) = s'(t - u + kT)$. Now, interchange the integrals:
\[
\int_0^T \int_0^L u^{-\alpha} s'(t - u) \, du \, dt = \int_0^L u^{-\alpha} \left( \int_0^T s'(t - u) \, dt \right) du.
\]
The inner integral evaluates to:
\[
\int_0^T s'(t - u) \, dt = s(T - u) - s(-u) = 0,
\]
where the last equality follows from $s$ being $T$-periodic. Thus, the original integral vanishes.\\[-1em]

\noindent \textbf{Case 2: $\alpha = 1$:} The CFDS reduces to the ordinary derivative:
\[
{}_L^{MC} D_t^1 s(t) = s'(t),
\]
and the integral becomes:
\[
\int_0^T s'(t) \, dt = s(T) - s(0) = 0,
\]
again by periodicity. This completes the proof.
\end{proof}

To further verify Lemma \ref{lem:1} numerically, we computed\\ $\int_0^T {}_L^{\text{MC}} D_t^\alpha s(t) \, dt$ numerically for the test problem studied in Section \ref{sec:NTP1} using $N = 300$ and the data in Table \ref{tab:parameters}. The approximation was based on \cite[Formula (4.7)]{elgindy2023new}:
\[
\int_0^T {}_L^{\text{MC}} D_t^\alpha s(t) \, dt \approx \frac{T}{N} \sum_{i=0}^{N-1} {}_L^{\text{MC}} D_{t_i}^\alpha s(t),
\]
where $t_0, t_1, \ldots, t_{N-1}$ are the collocation points. The computed value was about $2.055 \times 10^{-16}$, which aligns perfectly with Lemma \ref{lem:1}, and confirm that the integral of the CFDS over one period vanishes as theoretically predicted.

\begin{remark}
The vanishing of the integral of the CFDS over one period for periodic absolutely continuous functions is analogous to the classical result for standard derivatives, where $\int_0^T s'(t)\,dt = 0$ for any $T$-periodic differentiable function $s$. However, this property does not generally hold for other FD definitions. For instance, the Riemann-Liouville FD of periodic functions typically does not satisfy this zero-integral property due to its distinct kernel and non-local memory properties, which differ from those of the CFDS and do not preserve periodicity in the same way. Similarly, the classical Caputo FD, while sharing the same kernel as the CFDS, differs in its integration domain, using a fixed initial point rather than the sliding memory window of the CFDS. This difference in integration domains disrupts the zero-integral property for periodic functions in the classical Caputo case, whereas the CFDS's sliding memory aligns with periodicity to maintain this property.
\end{remark}

\section{Lemma on State Convexity}
\label{app:convexity}
\begin{lemma}[Non-convexity of State Trajectories]\label{lem:state_convexity}
Let $K Y \neq 1$, and suppose that $s_1$ and $s_2$ are two distinct periodic solutions of the FDE \eqref{eq:reducedFDE} corresponding to distinct controls $D_1$ and $D_2$ in $\C{D}$, respectively. Then for any $\lambda \in (0,1)$, the convex combination $s_\lambda = \lambda s_1 + (1-\lambda) s_2$ cannot be a solution of \eqref{eq:reducedFDE} corresponding to $D_\lambda = \lambda D_1 + (1-\lambda)D_2$.
\end{lemma}
\begin{proof}
We integrate the FDE \eqref{eq:reducedFDE} over $[0, T]$. By Lemma \ref{lem:1}, 
\[\int_0^T {}_L^{\text{MC}} D_t^\alpha s(t) \, dt = 0,\]
so:
\[
\int_0^T [D(t) - \nu(s(t))](\Sin - s(t)) \, dt = 0,
\]
implying the \textit{``integral balance equation''}:
\begin{equation}\label{eq:IntIdent1}
\left[D (\Sin - s)\right]\av = \left[\nu(s) (\Sin - s)\right]\av.
\end{equation}
For the steady-state, this simplifies into:
\begin{equation}\label{eq:IntIdent2}
\bD (\Sin - \bs) = \nu(\bs) (\Sin - \bs).
\end{equation}
Substitute $s_1$ and $s_2$ into the integral identity \eqref{eq:IntIdent1}:
\begin{equation}\label{eq:TIAC1_app}
\int_0^T \nu(s_i(t))(\Sin - s_i(t)) \, dt = \int_0^T D_i(t)(\Sin - s_i(t)) \, dt, \quad i = 1,2.
\end{equation}
Assume, for contradiction, that $s_\lambda(t) = \lambda s_1(t) + (1 - \lambda) s_2(t)$ for some $\lambda \in (0,1)$. From the dynamics and the convex combination of controls, the following must hold:
\begin{align}
&\int_0^T \nu(s_\lambda(t))(\Sin - s_\lambda(t)) \, dt = \int_0^T D_\lambda(t)(\Sin - s_\lambda(t)) \, dt \notag\\
&= \lambda \int_0^T D_1(t)(\Sin - s_\lambda(t)) \, dt + (1 - \lambda) \int_0^T D_2(t)(\Sin - s_\lambda(t)) \, dt.\label{eq:sddvsdb1} 
\end{align}
The condition $K Y \neq 1$ ensures $\nu$ is strictly convex/concave on $[0, \Sin]$, as
\[
\nu''(s) = \frac{2 K Y \mu_{\max} \Sin (K Y - 1)}{\left[ K Y (\Sin - s) + s \right]^3} \neq 0.
\]
So, by Jensen's inequality and the assumption that $s_1(t) \ne s_2(t)$ on a set of positive measure, we have
\[
\nu(s_\lambda(t)) \neq \lambda \nu(s_1(t)) + (1 - \lambda) \nu(s_2(t)) \quad \text{for a.e. } t.
\]
Multiplying both sides by $\Sin - s_\lambda(t) > 0$ and integrating, taking into account identities \eqref{eq:TIAC1_app}, gives:
\begin{gather*}
\int_0^T \nu(s_\lambda(t))(\Sin - s_\lambda(t)) \, dt \neq \lambda \int_0^T \nu(s_1(t))(\Sin - s_\lambda(t)) \, dt\\
+ (1 - \lambda) \int_0^T \nu(s_2(t))(\Sin - s_\lambda(t)) \, dt\\
= \lambda \int_0^T D_1(t)(\Sin - s_\lambda(t)) \, dt + (1 - \lambda) \int_0^T D_2(t)(\Sin - s_\lambda(t)) \, dt\\
= \int_0^T D_{\lambda}(t)(\Sin - s_\lambda(t)) \, dt,
\end{gather*} 
which contradicts identity \eqref{eq:sddvsdb1}. Hence, $s_\lambda$ cannot be a solution corresponding to $D_\lambda$.
\end{proof}

\section{Local Stability of the Equilibrium}
\label{app:LSE}
\begin{lemma}\label{lem:etey1}
Let $\alpha \in (0, 1)$ and $L > 0$, and consider the FDE
\[
{}^{\text{MC}}_{L}D_t^\alpha z(t) = -k z(t), \quad k > 0.
\]
Then the solution to this FDE can be expressed in the exponential form $z(t) = z(0) e^{-\lambda t}$ for some $\lambda > 0$.
\end{lemma}

\begin{proof}
Substituting $z(t) = z(0) e^{-\lambda t}$ into the definition of the CFDS yields:
\[
{}^{\text{MC}}_{L}D_t^\alpha z(t) = \frac{1}{\Gamma(1-\alpha)} \int_{t-L}^t (t - \tau)^{-\alpha} \left(-\lambda z(0) e^{-\lambda \tau}\right) \, d\tau.
\]
Changing variables via $u = t - \tau$ gives
\[
{}^{\text{MC}}_{L}D_t^\alpha z(t) = -\frac{\lambda z(0) e^{-\lambda t}}{\Gamma(1-\alpha)} \int_0^L u^{-\alpha} e^{\lambda u} \, du.
\]
Equating this with the right-hand side of the differential equation,
\[
-k z(t) = -k z(0) e^{-\lambda t},
\]
yields
\begin{equation}\label{eq:csnvbsdj1}
\lambda \int_0^L u^{-\alpha} e^{\lambda u} \, du = k\,\Gamma(1-\alpha).
\end{equation}
Since $k\,\Gamma(1-\alpha) > 0$, any solution $\lambda$ to Eq. \eqref{eq:csnvbsdj1} must be strictly positive. Now, it remains to verify whether Eq. \eqref{eq:csnvbsdj1} admits a solution or not. To this end, define
\[ \psi(\lambda):= \lambda \int_0^L u^{-\alpha} e^{\lambda u} du, \]
so Eq. \eqref{eq:csnvbsdj1} reads:
\begin{equation}\label{eq:csnvbsdj2}
\psi(\lambda) = k\,\Gamma(1-\alpha).
\end{equation}
Notice that $\psi$ is continuous for all $\lambda \in \mathbb{R}$, since the integrand is continuous in both $u \in (0, L]$ and $\lambda$. When $\lambda > 0$ increases, $e^{\lambda u}$ increases, so $\psi$ is strictly increasing on $(0, \infty)$. Also, 
\[\lim_{\lambda \to 0^+} \psi(\lambda) = 0,\quad \lim_{\lambda \to \infty} \psi(\lambda) = \infty,\]
because $e^{\lambda u}$ dominates the integral. Since $\psi$ is continuous, strictly increasing, and spans the interval $(0, \infty)$, Eq. \eqref{eq:csnvbsdj2} has a unique solution $\lambda > 0$ for any given $k > 0$, $\alpha \in (0, 1)$, and $L > 0$, by the Intermediate Value Theorem.
\end{proof}

\begin{theorem}\label{thm:local_stability}
Consider the FOCS governed by the FDE \eqref{eq:reducedFDE} with $\bD < \mu_{\max}$ being the constant dilution rate. Let $\bs$ given by \eqref{eq:equil2} be the equilibrium satisfying $\nu(\bs) = \bD$. Then the equilibrium $\bs$ is locally asymptotically stable, with perturbations $z(t) = s(t) - \bs$ decaying as $z(t) \sim e^{-\lambda t}$, for some $\lambda > 0$. Furthermore, for the initial condition $s(0) = s_0^- < \bs$, the solution $s$ approaches $\bs$ monotonically from below, does not reach $\bs$ in finite time, and cannot satisfy the PBC \eqref{eq:PBC1} for any $T > 0$. Similarly, for $s(0) = s_0^+ > \bs$, $s$ approaches $\bs$ monotonically from above, does not reach $\bs$ in finite time, and cannot satisfy \eqref{eq:PBC1} for any $T > 0$.
\end{theorem}
\begin{proof}
To analyze the local stability of $\bs$, we linearize the FDE \eqref{eq:reducedFDE} around the equilibrium. Define the perturbation $z(t) = s(t) - \bs$. Since $\bs$ is constant, ${}_L^{\text{MC}} D_t^\alpha s(t) = {}_L^{\text{MC}} D_t^\alpha z(t)$. Define $f$ as:
\[
f(s(t)) = [\bD - \nu(s(t))] (\Sin - s(t)).
\]
Expand $\nu(s(t)) = \nu(\bs + z(t))$ around $\bs$:
\[
\nu(s(t)) \approx \nu(\bs) + \nu'(\bs) z(t).
\]
Since $\nu(\bs) = \bD$, we have:
\[
\bD - \nu(s(t)) \approx -\nu'(\bs) z(t),
\]
so
\[
f(s(t)) \approx [-\nu'(\bs) z(t)] [(\Sin - \bs) - z(t)] \approx -\nu'(\bs) (\Sin - \bs) z(t),
\]
neglecting higher-order terms. Thus, the linearized FDE is:
\[
{}_L^{\text{MC}} D_t^\alpha z(t) = -k z(t), \quad k = \vartheta^{1-\alpha} \nu'(\bs) (\Sin - \bs) > 0,
\]
with the solution:
\[
z(t) = z(0) e^{-\lambda t},
\]
for some $\lambda > 0$, by Lemma \ref{lem:etey1}. Now, consider the two cases for the initial condition:

\textbf{Case 1: $s(0) = s_0^- < \bs$.} Here, $z(0) = s_0^- - \bs < 0$, so $z(t) < 0$, and $s(t) = \bs + z(t) < \bs$. Since 
\begin{equation*}
\nu'(s) = \frac{KY \mu_{\max} \Sin}{(KY (\Sin - s) + s)^2} > 0,
\end{equation*}
then $\nu$ is strictly increasing on $s([0, T])$. Therefore, $\nu(s(t)) < \nu(\bs) = \bD$, so $\bD - \nu(s(t)) > 0$, implying ${}_L^{\text{MC}} D_t^\alpha s(t) > 0$. This shows that $s$ is monotonically increasing toward $\bs$. The decay $|z(t)| \sim |z(0)| e^{-\lambda t}$ ensures $z(t) \neq 0$ for finite $t$, so $s(t) \neq \bs$. For periodicity, \eqref{eq:PBC1} requires $s(T) = s(0)$:
\[
s(T) = \bs + (s_0^- - \bs) e^{-\lambda t}.
\]
Since $e^{-\lambda t} \in (0, 1)$ and $s_0^- - \bs < 0$, $|z(T)| = |s_0^- - \bs| e^{-\lambda t} < |s_0^- - \bs| = |z(0)|$, so $z(T) = (s_0^- - \bs) e^{-\lambda t} > s_0^- - \bs = z(0)$, implying $s(T) > s_0^-$, violating $s(T) = s(0)$.

\textbf{Case 2: $s(0) = s_0^+ > \bs$.} Here, $z(0) = s_0^+ - \bs > 0$, so $z(t) > 0$, and $s(t) = \bs + z(t) > \bs$. Since $\nu(s(t)) > \nu(\bs) = \bD$, $\bD - \nu(s(t)) < 0$, so we have ${}_L^{\text{MC}} D_t^\alpha s(t) < 0$, implying $s(t)$ is monotonically decreasing toward $\bs$. The decay $z(t) \sim z(0) e^{-\lambda t}$ ensures $z(t) \neq 0$ for finite $t$, so $s(t) \neq \bs$. For periodicity:
\[
s(T) = \bs + (s_0^+ - \bs) e^{-\lambda t}.
\]
Since $e^{-\lambda t} \in (0, 1)$, $s(T) < s_0^+$, violating $s(T) = s(0)$.

In both cases, $|z(t)| \to 0$ as $t \to \infty$, confirming $\bs$ is locally asymptotically stable. The exponential decay $e^{-\lambda t}$ prevents $s(t)$ from reaching $\bs$ in finite time, and monotonicity prevents periodicity unless $s(0) = \bs$, where $s(t) \equiv \bs$.
\end{proof}

\section{Orbital Stability of the OPS}
\label{sec:orbital_stability}
Although Theorem \ref{thm:local_stability} establishes local asymptotic stability of equilibria under constant dilution rates, the OPC $D^*$ derived in Section \ref{sec:optimal_control} is non-constant and bang-bang, switching abruptly between $D_{\min}$ and $D_{\max}$ at finite times $\xi_k$, $k=1,\dots,2m$, for some $m \in \MB{N}$. This structure implies that between switching events, the system evolves under constant dilution rate dynamics. On each interval $[\xi_k, \xi_{k+1})$, the RFOCP dynamics reduces to:
\begin{equation}
{}_L^{\text{MC}} D_t^\alpha s(t) = \vartheta^{1-\alpha} [D_i - \nu(s(t))] (\Sin - s(t)), \quad D_i \in \{D_{\min}, D_{\max}\}.
\end{equation}
Let $\bs_i$ satisfy $\nu(\bs_i) = D_i$. Then Theorem \ref{thm:local_stability} guarantees that perturbations $z(t) = s(t) - \bs_i$ decay exponentially:
\begin{equation}
|z(t)| \leq C e^{-\lambda t}, \quad \lambda > 0,
\end{equation}
on each segment, for some $\lambda$ depending on $\alpha$, $L$, and system parameters.

However, the OPS $s^*(t)$ is not an equilibrium---it is a $T$-periodic trajectory that transiently approaches or departs from $\bs_i$ in a controlled manner. To analyze stability of $s^*(t)$, consider a perturbed trajectory $s(t; s_0)$ with initial condition $s(0) = s^*(0) + \delta: |\delta|$ small, evolving under the same switching control $D^*(t)$, and define the perturbation $z(t) = s(t; s_0) - s^*(t)$. On each constant-control segment, $z(t)$ satisfies the linearized FDE:
\begin{equation}\label{eq:ref121230}
{}_L^{\text{MC}} D_t^\alpha z(t) = \frac{\partial \mathcal{F}}{\partial s}(t, s^*(t)) \cdot z(t) + O(|z|^2),
\end{equation}
where 
\begin{equation}\label{eq:ref12123}
\mathcal{F}(t, s) = \vartheta^{1-\alpha} [D^*(t) - \nu(s)] (\Sin - s). 
\end{equation}
The linear coefficient is negative due to local stability of $\bs_i$ and convexity of $\nu$ (for $KY > 1$), ensuring contraction: $|z(t)|$ decreases on each segment.

Now, define the Poincar\'e map
\[
\C{P}: s(0) \mapsto s(T; s(0)),
\]
where $s(t; s_0)$ is the solution of the RFOCP dynamics:
\[
{}_L^{\text{MC}} D_t^\alpha s(t) = \vartheta^{1-\alpha} [D^*(t) - \nu(s(t))] (\Sin - s(t)),
\]
with fixed OPC $D^*(t)$ (bang-bang, $T$-periodic) and initial condition $s(0) = s_0$. Since the OPS $s^*(t)$ satisfies $s^*(t + T) = s^*(t)$ for all $t$, it follows that
\[
s^*(T) = s^*(0).
\]
By the definition of the solution flow under $D^*(t)$,
\[
s^*(T) = s(T; s^*(0)).
\]
Thus,
\[
\C{P}(s^*(0)) = s^*(0),
\]
so $s^*(0)$ is a fixed point of $\C{P}$. We now prove that $\C{P}$ is a contraction mapping in a neighborhood of $s^*(0)$, which implies orbital asymptotic stability of the OPS $s^*(t)$.

\subsubsection*{Proof of Contraction Mapping}
\label{subsec:POC1}
The OC $D^*(t)$ is bang-bang, switching between $D_{\min}$ and $D_{\max}$ at finitely many times
$0 \leq \xi_1 < \xi_2 < \cdots < \xi_{2m} < T$, with $\xi_{2m+1} := T + \xi_1$, and $D^* \equiv D_k \in \{D_{\min}, D_{\max}\}$ on each interval $I_k := [\xi_k, \xi_{k+1}),\; k = 1, \dots, 2m$. Let $\Delta t_k := \xi_{k+1} - \xi_k > 0$, so $\sum_{k=1}^{2m} \Delta t_k = T$. On each $I_k$, the system evolves under constant dilution rate $D_k$, and has the local equilibrium
\[
\bs_k \in (0, \Sin):\quad \nu(\bs_k) = D_k.
\]
By Theorem \ref{thm:local_stability}, perturbations around $\bs_k$ decay exponentially:
\[
|z(t)| \leq C_k e^{-\lambda_k (t - \xi_k)}, \quad t \in I_k,
\]
for any trajectory starting near $\bs_k$, with
\[
\lambda_k = \lambda(\alpha, L, \vartheta, \nu'(\bs_k), \Sin - \bs_k) > 0,
\]
and $C_k > 0$ (a constant depending only on the system parameters and the segment length). Consider now the two initial conditions:
\begin{enumerate}[label=(\roman*)]
    \item $s_1(0) = s^*(0)$, where $s_1$ evolves to $s_1(T) = s^*(T) = s^*(0)$.
    \item $s_2(0) = s^*(0) + \delta$, where $s_2$ evolves to $s_2(T) = \mathcal{P}(s^*(0) + \delta)$.
\end{enumerate}
Define the error trajectory
\[
z(t) := s_2(t) - s_1(t) = s(t; s^*(0) + \delta) - s^*(t):\quad z(0) = \delta.
\]
On $I_k$, both $s_1(t) = s^*(t)$ and $s_2(t)$ evolve under the same constant control $D_k$. The error $z(t)$ satisfies the linearized FDE \eqref{eq:ref121230} with $\C{F}$ as defined by \eqref{eq:ref12123}. Since $s^*(t) \in (0, \Sin)$ for all $t \in [0, T]$, the periodic reference trajectory lies strictly inside the open interval $(0, \Sin)$ and thus its image $\{s^*(t) : t \in [0, T]\}$ is contained in the compact interval $[0, \Sin]$. Moreover, on each sub-interval $I_k$, the OC is constant: $D^*(t) \equiv D_k$. Finally, the kinetic function $\nu \in C^2([0, \Sin])$. Under these assumptions, a bound on $\partial \C{F}/\partial s$ evaluated along the reference trajectory can be established by the following theorem.

\begin{theorem}\label{thm:jacobian-bound}
Let $s^*(t)$ be a continuous $T$-periodic solution of the system \eqref{eq:reducedFDE} with $\nu \in C^2([0, \Sin])$, $\alpha \in (0,1)$, and $\vartheta > 0$. For each subinterval $I_k \subset [0, T]$ with length $\Delta t_k > 0$, let
\[
\bs_k := \frac{1}{\Delta t_k} \int_{I_k} s^*(t) \, dt,
\]
be the temporal average of $s^*(t)$ over $I_k$. Then
\[
\left| \frac{\partial \C{F}}{\partial s}\left(t, s^*(t)\right) + \nu'(\bs_k) (\Sin - \bs_k) \right|
\leq M_k \, \|s^*(t) - \bs_k\|_\infty, \quad \forall t \in I_k,
\]
where
\[
M_k := \sup_{s \in [0, \Sin]} \left|\nu''(s)\right| < \infty,
\]
and
\[
\|s^* - \bs_k\|_\infty := \sup_{t \in I_k} |s^*(t) - \bs_k|.
\]
\end{theorem}
\begin{proof}
Evaluating the partial derivative of $\C{F}$ with respect to $s$ along the reference trajectory $s^*$ gives
\[
\frac{\partial \C{F}}{\partial s}\left(t, s^*(t)\right) = -\vartheta^{1-\alpha} \left[ \nu'\left(s^*(t)\right) (\Sin - s^*(t)) + D(t) - \nu\left(s^*(t)\right) \right].
\]
On each subinterval $I_k = [\xi_k, \xi_{k+1}]$, approximate $s^*(t) \approx \bs_k$ and $D(t) \approx D_k$, and define
\[
\kappa_k := \nu'(\bs_k) (\Sin - \bs_k) + D_k - \nu(\bs_k) > 0,
\]
so that the linearization error is $\displaystyle{\frac{\partial \C{F}}{\partial s}}\left(t, s^*(t)\right) + \vartheta^{1-\alpha} \kappa_k$. Define the auxiliary function
\[
h(t, s) := \nu'(s) (\Sin - s) + D(t) - \nu(s),
\]
so the error becomes
\[
\vartheta^{1-\alpha} \left( h(t, \bs_k) - h\left(t, s^*(t)\right) \right).
\]
Assume $D \in C([0, T])$ and $\nu \in C^2([0, \Sin])$. Then
\[
\frac{\partial h}{\partial s}(t, s) = \nu''(s) \, (\Sin - s) - 2 \nu'(s).
\]
Thus
\[
\left| \frac{\partial h}{\partial s}(t, s) \right| \leq M_k \, \Sin + 2 \sup_{s \in [0, \Sin]} |\nu'(s)| =: C_k < \infty,
\]
uniformly on $I_k$. By the Mean Value Theorem, for each $t \in I_k$ there exists $\xi_t$ between $s^*(t)$ and $\bs_k$ such that
\[
h\left(t, s^*(t)\right) - h(t, \bs_k) = \frac{\partial h}{\partial s}(t, \xi_t) \left( s^*(t) - \bs_k \right).
\]
Thus,
\[
\left| h\left(t, s^*(t)\right) - h(t, \bs_k) \right| \leq C_k \, |s^*(t) - \bs_k| \leq C_k \, \|s^* - \bs_k\|_{L^\infty(I_k)}.
\]
Therefore,
\[
\left| \frac{\partial \C{F}}{\partial s}\left(t, s^*(t)\right) + \vartheta^{1-\alpha} \kappa_k \right| \leq \vartheta^{1-\alpha} C_k \, \|s^* - \bs_k\|_{L^\infty(I_k)},
\]
where the uniform deviation
\[
\|s^* - \bs_k\|_{L^\infty(I_k)} := \sup_{t \in I_k} |s^*(t) - \bs_k|,
\]
is finite since $s^*$ is continuous on the compact interval $I_k$. Over all $2m$ subintervals,
\[
\max_{k=1,\dots,2m} \|s^* - \bs_k\|_{L^\infty(I_k)},
\]
is bounded by a constant depending only on the reference orbit $s^*(t)$, independent of perturbations. We can absorb $\vartheta^{1-\alpha} C_k$ and this maximum deviation into a new bounded constant $M_k$ (independent of $t$ and perturbations) to obtain
\[
\left| \frac{\partial \C{F}}{\partial s}\left(t, s^*(t)\right) + \vartheta^{1-\alpha} \kappa_k \right| \leq M_k \|z\|_{L^\infty(I_k)},
\]
for $z(t) = s(t) - s^*(t)$. Thus, for sufficiently small $\|z\|$, the nonlinear perturbation is negligible, and
\[
{}_L^{\text{MC}} D_t^\alpha z(t) \approx -\vartheta^{1-\alpha} \kappa_k z(t), \quad \kappa_k > 0.
\]
By Lemma \ref{lem:etey1}, the solution satisfies
\[
|z(t)| \leq |z(\xi_k)| e^{-\lambda_k (t - \xi_k)}, \quad t \in I_k,
\]
with $\lambda_k > 0$ such that $e^{-\lambda_k \Delta t_k} < 1$. Since $|z(0)| = |\delta|$, we have after segment $I_1$ (at $t = \xi_2$):
\[
|z(\xi_2)| \leq |z(\xi_1)| e^{-\lambda_1 \Delta t_1} \leq |\delta| e^{-\lambda_1 \Delta t_1}.
\]
After segment $I_2$ (at $t = \xi_3$):
\[
|z(\xi_3)| \leq |z(\xi_2)| e^{-\lambda_2 \Delta t_2} \leq |\delta| e^{-\lambda_1 \Delta t_1 -\lambda_2 \Delta t_2}.
\]
After all $2m$ segments (at $t = T$):
\[
|z(T)| \leq |\delta| \prod_{k=1}^{2m} e^{-\lambda_k \Delta t_k} = \rho\,|\delta|,
\]
where
\[
\rho := \exp\left( -\sum_{k=1}^{2m} \lambda_k \Delta t_k \right) < 1,
\]
since $\lambda_k, \Delta t_k > 0$. Thus,
\[
|\C{P}(s^*(0) + \delta) - s^*(T)| = |z(T)| \leq \rho |\delta|.
\]
Apply $\C{P}$ iteratively, and define the sequence of perturbations $(\delta_n)_{n=0}^\infty$, where each $\delta_n$ represents the deviation from the optimal periodic orbit after $n$ complete periods:
\begin{align*}
\delta_0 &:= \delta, \\
\delta_1 &:= \C{P}(s^*(0) + \delta) - s^*(T), \\
&\vdots \\
\delta_{n+1} &:= \C{P}^{n+1}(s^*(0) + \delta) - s^*(T) = \C{P}\bigl( \C{P}^n(s^*(0) + \delta) \bigr) - s^*(T),
\end{align*}
for $n \geq 0$. Then
\[
|\delta_n| \leq \rho^n |\delta| \to 0, \quad \text{as } n \to \infty,
\]
since $\rho < 1$. This shows that any trajectory starting sufficiently close to the OPS at $t=0$ returns closer to the orbit after each period $T$, and converges to the periodic orbit as $n \to \infty$. This establishes orbital asymptotic stability of the OPS $s^*(t)$.
\end{proof}
For a comprehensive exposition on stability analysis of fractional differential equations, Poincar\'{e} map analysis, and contraction arguments, the reader may consult \cite{khalil2002nonlinear,li2011survey,guckenheimer2013nonlinear}.

\begin{remark}
Numerical observations (Figure \ref{fig:Fig5}) suggest that lower $\alpha$ values, corresponding to stronger memory effects, may be associated with increased damping characteristics that require more frequent control switching to overcome system inertia. While Lemma \ref{lem:etey1} guarantees the existence of a positive damping coefficient $\lambda_k$ for each $\alpha \in (0,1)$, the precise functional relationship between $\alpha$ and $\lambda_k$ remains an open analytical question, though our simulations indicate that stronger memory generally improves system robustness at the cost of requiring more aggressive control intervention.
\end{remark}

\section{Perturbation Analysis}
\label{app:perturbation_analysis}
This section analyzes the relationship between $\nu\av$ and the average dilution rate $D\av$ for a $T$-periodic solution $s$ of the FDE \eqref{eq:reducedFDE}, when $\alpha \in (0,1)$ and $s(t) \in [0, \Sin)$. We use a perturbation approach around the steady-state to show that $\nu\av < D\av$ is only possible for non-constant solutions with small perturbations when $K Y < 1$.

Consider the steady-state where $D(t) = \bD = \nu(\bs)$, $s(t) = \bs$. Perturb the control and state as:
\[
\Deps(t) = \bD + \varepsilon v(t):\,v\av = 0,\quad \seps(t) = \bs + \varepsilon z(t),
\]
where $\varepsilon > 0$ is sufficiently small, and both $v$ and $z$ are $T$-periodic. The FDE becomes:
\[
{}_L^{\text{MC}} D_t^\alpha (\bs + \varepsilon z(t)) = \vartheta^{1-\alpha} [\bD + \varepsilon v(t) - \nu(\bs + \varepsilon z(t))] (\Sin - \bs - \varepsilon z(t)).
\]
Since ${}_L^{\text{MC}} D_t^\alpha \bs = 0$ and $\nu(\bs) = \bD$, expand $\nu(s)$:
\[
\nu(\bs + \varepsilon z(t)) \approx \nu(\bs) + \varepsilon \nu'(\bs) z(t) + \frac{\varepsilon^2}{2} \nu''(\bs) z^2(t).
\]
For the linear approximation, neglect $O(\varepsilon^2)$ terms:
\[
\varepsilon\,{}_L^{\text{MC}} D_t^\alpha z(t) \approx \vartheta^{1-\alpha} [\varepsilon v(t) - \varepsilon \nu'(\bs) z(t)] (\Sin - \bs).
\]
Integrate over $[0, T]$, noting that $\int_0^T {}_L^{\text{MC}} D_t^\alpha z(t) \, dt = 0$ by Lemma \ref{lem:1}:
\begin{equation}\label{eq:xbcvj1}
0 = \vartheta^{1-\alpha} \varepsilon\,(\Sin - \bs) \int_0^T [v(t) - \nu'(\bs) z(t)] \, dt.
\end{equation}
Since $v\av = 0$, Eq. \eqref{eq:xbcvj1} implies:
\[
-\nu'(\bs) \int_0^T z(t) \, dt = 0 \implies z_{\av} = 0,
\]
since $\nu' > 0$. Computing the difference:
\[
\Deps(t) - \nu(\seps(t)) \approx \varepsilon v(t) - \varepsilon \nu'(\bs) z(t).
\]
Therefore,
\begin{align*}
D\epsav - [\nu(\seps)]\av &= \frac{1}{T} \int_0^T [\Deps(t) - \nu(\seps(t))] \, dt\\
&\approx \varepsilon v\av - \varepsilon \nu'(\bs) z\av = 0.
\end{align*}
Thus, to first order, $D\epsav \approx [\nu(\seps)]\av$. Include the second-order term:
\[
\Deps(t) - \nu(\seps(t)) \approx \varepsilon v(t) - \varepsilon \nu'(\bs) z(t) - \frac{\varepsilon^2}{2} \nu''(\bs) z^2(t).
\]
Integrate:
\[
D\epsav - [\nu(\seps)]\av \approx -\frac{\varepsilon^2}{2} \nu''(\bs) \left(z^2\right)\av.
\]
If $K Y < 1$, the function $\nu$ is strictly concave on the interval $s \in [0, \Sin]$, implying that $\nu''(\bs) < 0$. For a non-constant function $z$, it follows that $\left(z^2\right)\av > 0$. Consequently,
\[
D\epsav - [\nu(\seps)]\av > 0,
\]
which implies that 
\begin{equation}\label{eq:ssvegdflkl1}
[\nu(\seps)]\av < D\epsav = \bD.
\end{equation}
Conversely, if $K Y \ge 1$, the function $\nu$ is convex on $s \in [0, \Sin]$, such that $\nu''(\bs) \ge 0$, with strict inequality when $K Y > 1$. Therefore,
\[
D\epsav - [\nu(\seps)]\av \le 0,
\]
indicating that 
\begin{equation}\label{eq:dskvbwk11121}
[\nu(\seps)]\av \ge D\epsav = \bD, 
\end{equation}
with equality holding when $K Y = 1$.

This perturbation analysis demonstrates that, for small, non-constant perturbations around the steady state, the average value of $\nu$ over one cycle is less than the average dilution rate $\bD$ if and only if $K Y < 1$.

\section{Derivation of the Right-Sided CFDS}
\label{sec:DRSCFDS1}
This appendix presents the formulation and justification of the right-sided CFDS of order $0 < \alpha < 1$, denoted by ${{}^{\text{MC}}_{L+}D_t^\alpha f}$, which is particularly useful for modeling forward-looking memory effects in fractional-order dynamical systems. The right-sided CFDS is the forward-time analogue of the left-sided CFDS ${}^{\text{MC}}_{L}D_t^\alpha f$ used mainly in this paper.

Let $L > 0$, and suppose $f \in W^{1,1}_{\text{loc}}([t, t+L])$. The right-sided CFDS is defined as
\begin{equation}\label{eq:LSCFDSM1}
{{}^{\text{MC}}_{L+}D_t^\alpha f} := -\frac{1}{\Gamma(1 - \alpha)} \int_t^{t+L} \frac{f'(\tau)}{(\tau - t)^{\alpha}}\, d\tau.
\end{equation}
This operator is well-defined a.e., since the kernel $(\tau - t)^{-\alpha} \in L^q(t, t+L)$ for all $q < 1/\alpha$, and the integrability of $f' \in L^1_{\text{loc}}([t,t+L])$ ensures the convergence of the integral. This operator definition is the finite-memory version of the classical right-sided Caputo FD of $f$ on the interval $[t, b]$, given by
\begin{equation}
{}^{\mathrm{C}}D_{b-}^\alpha f(t) := -\frac{1}{\Gamma(1 - \alpha)} \int_t^{b} \frac{f'(\tau)}{(\tau - t)^{\alpha}}\, d\tau,
\end{equation}
by replacing the upper limit $b$ with $t + L$. Notice that as $L \to b - t$, the sliding memory version recovers the classical right-sided Caputo FD:
\begin{equation}
\lim_{L \to b - t} {{}^{\text{MC}}_{L+}D_t^\alpha f(t)} = {}^{\mathrm{C}}D_t^\alpha f(t).
\end{equation}
Moreover, as $\alpha \to 1^{-}$, we recover the classical first-order derivative with a negative sign:
\begin{equation}
\lim_{\alpha \to 1^{-}} {{}^{\text{MC}}_{L+}D_t^\alpha f(t)} = -f'(t).
\end{equation}

\section{Numerical Optimization Techniques for Solving the RFOCP}
\label{app:NSMfStR1}
The continuous RFOCP is transformed into a finite-dimensional NLP problem through discretization using the FG-PS method \cite{elgindy2024fouriera,elgindy2024fourierb}. This method is particularly well-suited for problems with periodic solutions. 
The time domain $[0, T]$ is discretized into $N$ equispaced collocation points $t_j = j T/N$ for $j=0, \dots, N-1$. Collocation points can be chosen from various distributions (e.g., Chebyshev, Legendre, or equispaced). Here, equispaced points are selected due to their compatibility with Fourier expansions in the FG-PS method, enabling efficient FFT-based computations and natural handling of periodic boundary conditions. The state variables and control inputs are approximated by their values at these collocation points. The FD term is handled using a FG-PS-based integration matrix, which is pre-computed. This matrix transforms the FDE \eqref{eq:reducedFDE} into a system of algebraic equations. The discretized RFOCP is solved as a constrained NLP problem. The objective function to be minimized is the average substrate concentration $s\av$, which is directly computed from the discretized substrate values. The constraints include the dynamic equations of the system, the average dilution rate constraint, and the bounds on the state and control variables. The MATLAB \texttt{fmincon} function is employed as the optimization solver, configured to use the \texttt{sqp} algorithm. This choice is motivated by \texttt{sqp}'s effectiveness in solving constrained nonlinear optimization problems, particularly when the objective function is continuous and the formulation includes general nonlinear constraints and bound constraints. In this context, it provides high accuracy and robust constraint satisfaction, making it well-suited for the discretized RFOCP.

\subsection{Edge-Detection Control Correction}
\label{sec:EDCC}
After obtaining a predicted OC profile from \texttt{fmincon}, the MATLAB code applies an edge-detection method to refine the control, particularly for bang-bang type controls which are characterized by abrupt switches between their minimum and maximum values. This correction is crucial because numerical optimization methods, especially those based on pseudospectral collocation, can introduce Gibbs phenomenon artifacts around discontinuities, leading to poor representations of true bang-bang controls. The method employed here is based on the principles outlined in \cite{elgindy2023new} and further upgraded in \cite{elgindy2024optimal}.

The core idea of the edge-detection method is to accurately identify the switching points in the OC signal and then reconstruct a bang-bang control based on these detected points. This approach exploits the fact that the Gibbs phenomenon, while a numerical artifact, provides a strong indicator of the location of discontinuities through its characteristic overshoots and undershoots. As stated in \cite{elgindy2024optimal}, quoting \cite{elgindy2023new}:
\begin{quote}
`While Gibbs phenomenon is generally considered a demon that needs to be cast out, we shall demonstrate later that it is rather 'a blessing', in view of the current work, that can be constructively used to set up a robust adaptive algorithm. In particular, the over- and undershoots developed near a discontinuity in the event of a Gibbs phenomenon provide an excellent means of detecting one.'
\end{quote}

The MATLAB code implements this correction in the following steps:

\begin{description}
    \item[Step 1] The predicted OC, $D^*$, is used to compute its Fourier coefficients, which capture the global spectral information of $D$, including potential jump discontinuities.
    \item[Step 2] An edge-detection solver is invoked to estimate the discontinuity locations and reconstruct an approximate bang-bang control. This function analyzes the Fourier interpolant constructed from the coefficients and evaluates where significant changes in the pseudospectral profile occur. 
    \item[Step 3] Based on the estimated discontinuities and the approximated control, a corrected bang-bang control $D^*$ is generated. This reconstruction effectively eliminates Gibbs oscillations and yields a physically meaningful bang-bang structure.
\item[Step 4] The reconstructed OC values at the same collocation set of $N$ equispaced is then used as inputs to compute the corrected substrate concentration $s^*$ by solving the discretized FED \eqref{eq:reducedFDE} using MATLAB's \texttt{fsolve} solver, with the predicted substrate concentration values provided as initial guesses, closely following the predictor-corrector framework in \citet{elgindy2023new}.
\end{description}

This two-stage approach--predicted pseudospectral optimization followed by edge detection and reconstruction--yields a robust and accurate method for solving the RFOCP that admits bang-bang solutions. It addresses the deficiencies of conventional pseudospectral methods in resolving discontinuities. 

\begin{remark}
Unlike the correction stage in the Fourier-Gegenbauer predictor-corrector method developed in \cite{elgindy2023new}, which requires collocation at shifted Gegenbauer-Gauss points to enable the use of barycentric shifted Gegenbauer quadratures, the present approach offers greater flexibility. Specifically, the current correction stage permits collocation at the same equally spaced points used in the prediction stage, as the FG-PS method can approximate the CFDS at those equispaced collocation points within the solution domain. In contrast, the use of shifted Gegenbauer quadratures in \cite{elgindy2023new} confines integration to the solution values at the shifted Gegenbauer--Gauss nodes.
\end{remark}

\section{Solution Algorithm}
\label{app:algorithm}
The following algorithm provides a comprehensive framework for solving the FOCP, beginning with the original 2D system and proceeding through numerical optimization.

\begin{algorithm}[H]
\caption{Solution Algorithm for the FOCS}
\label{alg:main}
\begin{algorithmic}[1]
\REQUIRE Original 2D FOCS: Eqs. \eqref{eq:sysdyn1}--\eqref{eq:sysdyn2} with constraints \eqref{eq:IntConsD1}--\eqref{eq:Boundd3}, \eqref{eq:PBC1}--\eqref{eq:PBC3}.
\ENSURE System parameters satisfy existence conditions (Theorem \ref{thm:existence}).

\textbf{Phase 1: System Reduction.}
\STATE Apply Transformation \eqref{eq:Transs1}.
\STATE Derive the 1D FDE \eqref{eq:trans2}.
\STATE Obtain the biomass relation \eqref{eq:biomass2}.
\STATE Reduce to the 1D RFOCP dynamics \eqref{eq:reducedFDE}.

\textbf{Phase 2: Numerical Solution.}
\STATE Discretize the time domain $[0,T]$ into $N$ equispaced collocation points.
\STATE Approximate the CFDS using FG-PS integration matrix.
\STATE Formulate and solve the constrained NLP using \texttt{fmincon} with SQP algorithm.
\STATE Apply edge-detection correction for bang-bang control reconstruction.
\STATE Interpolate solutions to $M$ points for high-resolution visualization.

\RETURN Optimal solution $(s^*, x^*, D^*)$ and performance metric $J(D^*)$.
\end{algorithmic}
\end{algorithm}

\bibliographystyle{model1-num-names}
\bibliography{Bib}

\begin{thebibliography}{41}
\expandafter\ifx\csname natexlab\endcsname\relax\def\natexlab#1{#1}\fi
\providecommand{\bibinfo}[2]{#2}
\ifx\xfnm\relax \def\xfnm[#1]{\unskip,\space#1}\fi
\bibitem[{Bayen et~al.(2020)Bayen, Rapaport, and Tani}]{bayen2020improvement}
\bibinfo{author}{T.~Bayen}, \bibinfo{author}{A.~Rapaport},
  \bibinfo{author}{F.~Z. Tani},
\newblock \bibinfo{title}{Improvement of performances of the chemostat used for
  continuous biological water treatment with periodic controls},
\newblock \bibinfo{journal}{Automatica} \bibinfo{volume}{121}
  (\bibinfo{year}{2020}) \bibinfo{pages}{109199}.
\bibitem[{Elgindy(2025)}]{elgindy2025sustainable}
\bibinfo{author}{K.~T. Elgindy},
\newblock \bibinfo{title}{Sustainable water treatment through fractional-order
  chemostat modeling with sliding memory and periodic boundary conditions: {A}
  mathematical framework for clean water and sanitation},
\newblock \bibinfo{journal}{arXiv preprint arXiv:2506.04420}
  (\bibinfo{year}{2025}).
\bibitem[{Elgindy(2023)}]{elgindy2023new}
\bibinfo{author}{K.~T. Elgindy},
\newblock \bibinfo{title}{New optimal periodic control policy for the optimal
  periodic performance of a chemostat using a {F}ourier--{G}egenbauer-based
  predictor-corrector method},
\newblock \bibinfo{journal}{Journal of Process Control} \bibinfo{volume}{127}
  (\bibinfo{year}{2023}) \bibinfo{pages}{102995}.
\bibitem[{Tarasov(2011)}]{tarasov2011fractional}
\bibinfo{author}{V.~E. Tarasov}, \bibinfo{title}{Fractional dynamics:
  {A}pplications of fractional calculus to dynamics of particles, fields and
  media}, \bibinfo{publisher}{Springer Science \& Business Media},
  \bibinfo{year}{2011}.
\bibitem[{Magin(2004)}]{magin2004fractional}
\bibinfo{author}{R.~Magin},
\newblock \bibinfo{title}{Fractional calculus in bioengineering, part 1},
\newblock \bibinfo{journal}{Critical Reviews in Biomedical Engineering}
  \bibinfo{volume}{32} (\bibinfo{year}{2004}).
\bibitem[{Caponetto et~al.(2010)Caponetto, Dongola, Fortuna, and
  Petras}]{caponetto2010fractional}
\bibinfo{author}{R.~Caponetto}, \bibinfo{author}{G.~Dongola},
  \bibinfo{author}{L.~Fortuna}, \bibinfo{author}{I.~Petras},
  \bibinfo{title}{Fractional order systems: {M}odeling and control
  applications}, volume~\bibinfo{volume}{72}, \bibinfo{publisher}{World
  Scientific}, \bibinfo{year}{2010}.
\bibitem[{Chakraverty et~al.(2023)Chakraverty, Jena, and
  Jena}]{chakraverty2023time}
\bibinfo{author}{S.~Chakraverty}, \bibinfo{author}{R.~M. Jena},
  \bibinfo{author}{S.~K. Jena},
\newblock \bibinfo{title}{Time-fractional order biological systems with
  uncertain parameters},
\newblock \bibinfo{journal}{Synthesis Lectures on Mathematics and Statistics}
  (\bibinfo{year}{2023}).
\bibitem[{Loudahi et~al.(2025)Loudahi, Ali, Yuan, Ahmad, Amin, and
  Wei}]{loudahi2025fractal}
\bibinfo{author}{L.~Loudahi}, \bibinfo{author}{A.~Ali},
  \bibinfo{author}{J.~Yuan}, \bibinfo{author}{J.~Ahmad}, \bibinfo{author}{L.~G.
  Amin}, \bibinfo{author}{Y.~Wei},
\newblock \bibinfo{title}{Fractal--fractional analysis of a water pollution
  model using fractional derivatives},
\newblock \bibinfo{journal}{Fractal and Fractional} \bibinfo{volume}{9}
  (\bibinfo{year}{2025}) \bibinfo{pages}{321}.
\bibitem[{Du et~al.(2013)Du, Wang, and Hu}]{du2013measuring}
\bibinfo{author}{M.~Du}, \bibinfo{author}{Z.~Wang}, \bibinfo{author}{H.~Hu},
\newblock \bibinfo{title}{Measuring memory with the order of fractional
  derivative},
\newblock \bibinfo{journal}{Scientific reports} \bibinfo{volume}{3}
  (\bibinfo{year}{2013}) \bibinfo{pages}{3431}.
\bibitem[{Ionescu et~al.(2017)Ionescu, Lopes, Copot, Machado, and
  Bates}]{ionescu2017role}
\bibinfo{author}{C.~Ionescu}, \bibinfo{author}{A.~Lopes},
  \bibinfo{author}{D.~Copot}, \bibinfo{author}{J.~T. Machado},
  \bibinfo{author}{J.~H. Bates},
\newblock \bibinfo{title}{The role of fractional calculus in modeling
  biological phenomena: {A} review},
\newblock \bibinfo{journal}{Communications in Nonlinear Science and Numerical
  Simulation} \bibinfo{volume}{51} (\bibinfo{year}{2017})
  \bibinfo{pages}{141--159}.
\bibitem[{Elgindy(2024{\natexlab{a}})}]{elgindy2024fouriera}
\bibinfo{author}{K.~T. Elgindy},
\newblock \bibinfo{title}{Fourier--{G}egenbauer pseudospectral method for
  solving periodic fractional optimal control problems},
\newblock \bibinfo{journal}{Mathematics and Computers in Simulation}
  \bibinfo{volume}{225} (\bibinfo{year}{2024}{\natexlab{a}})
  \bibinfo{pages}{148--164}.
\bibitem[{Elgindy(2024{\natexlab{b}})}]{elgindy2024fourierb}
\bibinfo{author}{K.~T. Elgindy},
\newblock \bibinfo{title}{Fourier--{G}egenbauer pseudospectral method for
  solving periodic higher-order fractional optimal control problems},
\newblock \bibinfo{journal}{International Journal of Computational Methods}
  \bibinfo{volume}{21} (\bibinfo{year}{2024}{\natexlab{b}})
  \bibinfo{pages}{2450015}.
\bibitem[{Elgindy(2024{\natexlab{c}})}]{elgindy2024fourier}
\bibinfo{author}{K.~T. Elgindy},
\newblock \bibinfo{title}{Fourier-gegenbauer pseudospectral method for solving
  time-dependent one-dimensional fractional partial differential equations with
  variable coefficients and periodic solutions},
\newblock \bibinfo{journal}{Mathematics and Computers in Simulation}
  \bibinfo{volume}{218} (\bibinfo{year}{2024}{\natexlab{c}})
  \bibinfo{pages}{544--555}.
\bibitem[{Podlubny(1998)}]{podlubny1998fractional}
\bibinfo{author}{I.~Podlubny}, \bibinfo{title}{Fractional differential
  equations: {A}n introduction to fractional derivatives, fractional
  differential equations, to methods of their solution and some of their
  applications}, volume \bibinfo{volume}{198}, \bibinfo{publisher}{elsevier},
  \bibinfo{year}{1998}.
\bibitem[{Kilbas(2006)}]{kilbas2006theory}
\bibinfo{author}{A.~Kilbas},
\newblock \bibinfo{title}{Theory and applications of fractional differential
  equations},
\newblock \bibinfo{journal}{North-Holland Mathematics Studies}
  \bibinfo{volume}{204} (\bibinfo{year}{2006}).
\bibitem[{Bourafa et~al.(2021)Bourafa, Abdelouahab, and
  Lozi}]{bourafa2021periodic}
\bibinfo{author}{S.~Bourafa}, \bibinfo{author}{M.~S. Abdelouahab},
  \bibinfo{author}{R.~Lozi},
\newblock \bibinfo{title}{On periodic solutions of fractional-order
  differential systems with a fixed length of sliding memory},
\newblock \bibinfo{journal}{Journal of Innovative Applied Mathematics and
  Computational Sciences} \bibinfo{volume}{1} (\bibinfo{year}{2021})
  \bibinfo{pages}{64--78}.
\bibitem[{Gokhale et~al.(2021)Gokhale, Giaimo, and Remigi}]{gokhale2021memory}
\bibinfo{author}{C.~S. Gokhale}, \bibinfo{author}{S.~Giaimo},
  \bibinfo{author}{P.~Remigi},
\newblock \bibinfo{title}{Memory shapes microbial populations},
\newblock \bibinfo{journal}{PLoS computational biology} \bibinfo{volume}{17}
  (\bibinfo{year}{2021}) \bibinfo{pages}{e1009431}.
\bibitem[{Faigenbaum-Romm et~al.(2025)Faigenbaum-Romm, Yedidi, Gefen,
  Katsowich-Nagar, Aroeti, Ronin, Bar-Meir, Rosenshine, and
  Balaban}]{faigenbaum2025uncovering}
\bibinfo{author}{R.~Faigenbaum-Romm}, \bibinfo{author}{N.~Yedidi},
  \bibinfo{author}{O.~Gefen}, \bibinfo{author}{N.~Katsowich-Nagar},
  \bibinfo{author}{L.~Aroeti}, \bibinfo{author}{I.~Ronin},
  \bibinfo{author}{M.~Bar-Meir}, \bibinfo{author}{I.~Rosenshine},
  \bibinfo{author}{N.~Q. Balaban},
\newblock \bibinfo{title}{Uncovering phenotypic inheritance from single cells
  with {M}icrocolony-seq},
\newblock \bibinfo{journal}{Cell} \bibinfo{volume}{188} (\bibinfo{year}{2025})
  \bibinfo{pages}{5313--5331}.
\bibitem[{Khalighi et~al.(2022)Khalighi, Sommeria-Klein, Gonze, Faust, and
  Lahti}]{khalighi2022quantifying}
\bibinfo{author}{M.~Khalighi}, \bibinfo{author}{G.~Sommeria-Klein},
  \bibinfo{author}{D.~Gonze}, \bibinfo{author}{K.~Faust},
  \bibinfo{author}{L.~Lahti},
\newblock \bibinfo{title}{Quantifying the impact of ecological memory on the
  dynamics of interacting communities},
\newblock \bibinfo{journal}{PLoS computational biology} \bibinfo{volume}{18}
  (\bibinfo{year}{2022}) \bibinfo{pages}{e1009396}.
\bibitem[{Amirian et~al.(2022)Amirian, Irwin, and
  Finkel}]{amirian2022extending}
\bibinfo{author}{M.~M. Amirian}, \bibinfo{author}{A.~J. Irwin},
  \bibinfo{author}{Z.~V. Finkel},
\newblock \bibinfo{title}{Extending the monod model of microbal growth with
  memory},
\newblock \bibinfo{journal}{Frontiers in Marine Science} \bibinfo{volume}{9}
  (\bibinfo{year}{2022}) \bibinfo{pages}{963734}.
\bibitem[{Rumbaugh and Sauer(2020)}]{rumbaugh2020biofilm}
\bibinfo{author}{K.~P. Rumbaugh}, \bibinfo{author}{K.~Sauer},
\newblock \bibinfo{title}{Biofilm dispersion},
\newblock \bibinfo{journal}{Nature Reviews Microbiology} \bibinfo{volume}{18}
  (\bibinfo{year}{2020}) \bibinfo{pages}{571--586}.
\bibitem[{Raboni et~al.(2013)Raboni, Torretta, and
  Urbini}]{raboni2013influence}
\bibinfo{author}{M.~Raboni}, \bibinfo{author}{V.~Torretta},
  \bibinfo{author}{G.~Urbini},
\newblock \bibinfo{title}{Influence of strong diurnal variations in sewage
  quality on the performance of biological denitrification in small community
  wastewater treatment plants ({WWTPs})},
\newblock \bibinfo{journal}{Sustainability} \bibinfo{volume}{5}
  (\bibinfo{year}{2013}) \bibinfo{pages}{3679--3689}.
\bibitem[{Clarke(2013)}]{clarke2013functional}
\bibinfo{author}{F.~Clarke}, \bibinfo{title}{Functional analysis, calculus of
  variations and optimal control}, volume \bibinfo{volume}{264},
  \bibinfo{publisher}{Springer}, \bibinfo{year}{2013}.
\bibitem[{Ahmed and Wang(2021)}]{ahmed2021optimal}
\bibinfo{author}{N.~Ahmed}, \bibinfo{author}{S.~Wang},
\newblock \bibinfo{title}{Optimal control: Existence theory},
\newblock in: \bibinfo{booktitle}{Optimal Control of Dynamic Systems Driven by
  Vector Measures: Theory and Applications}, \bibinfo{publisher}{Springer},
  \bibinfo{year}{2021}, pp. \bibinfo{pages}{109--149}.
\bibitem[{Liberzon(2011)}]{liberzon2011calculus}
\bibinfo{author}{D.~Liberzon}, \bibinfo{title}{Calculus of variations and
  optimal control theory: {A} concise introduction},
  \bibinfo{publisher}{Princeton university press}, \bibinfo{year}{2011}.
\bibitem[{Agrawal(2004)}]{agrawal2004general}
\bibinfo{author}{O.~P. Agrawal},
\newblock \bibinfo{title}{A general formulation and solution scheme for
  fractional optimal control problems},
\newblock \bibinfo{journal}{Nonlinear Dynamics} \bibinfo{volume}{38}
  (\bibinfo{year}{2004}) \bibinfo{pages}{323--337}.
\bibitem[{Kamocki(2014)}]{kamocki2014pontryagin}
\bibinfo{author}{R.~Kamocki},
\newblock \bibinfo{title}{Pontryagin maximum principle for fractional ordinary
  optimal control problems},
\newblock \bibinfo{journal}{Mathematical Methods in the Applied Sciences}
  \bibinfo{volume}{37} (\bibinfo{year}{2014}) \bibinfo{pages}{1668--1686}.
\bibitem[{Sahabi and Yazdani~Cherati(2024)}]{sahabi2024fractional}
\bibinfo{author}{M.~Sahabi}, \bibinfo{author}{A.~Yazdani~Cherati},
\newblock \bibinfo{title}{Fractional pseudospectral schemes with applications
  to fractional optimal control problems},
\newblock \bibinfo{journal}{Journal of Mathematics} \bibinfo{volume}{2024}
  (\bibinfo{year}{2024}) \bibinfo{pages}{9917116}.
\bibitem[{Yang et~al.(2024)Yang, Skandari, and Zhang}]{yang2024pseudospectral}
\bibinfo{author}{Y.~Yang}, \bibinfo{author}{M.~N. Skandari},
  \bibinfo{author}{J.~Zhang},
\newblock \bibinfo{title}{A pseudospectral method for continuous-time nonlinear
  fractional programming},
\newblock \bibinfo{journal}{Filomat} \bibinfo{volume}{38}
  (\bibinfo{year}{2024}) \bibinfo{pages}{1947--1961}.
\bibitem[{Kheirkhah and Hajipour(2023)}]{kheirkhah2023legendre}
\bibinfo{author}{F.~Kheirkhah}, \bibinfo{author}{M.~Hajipour},
\newblock \bibinfo{title}{The {L}egendre pseudospectral method for a
  time-fractional optimal control problem},
\newblock in: \bibinfo{booktitle}{Proceeding of The 12th International Seminar
  on Linear Algebra and its Applications}, p. \bibinfo{pages}{102}.
\bibitem[{Ejlali and Hosseini(2017)}]{ejlali2017pseudospectral}
\bibinfo{author}{N.~Ejlali}, \bibinfo{author}{S.~M. Hosseini},
\newblock \bibinfo{title}{A pseudospectral method for fractional optimal
  control problems},
\newblock \bibinfo{journal}{Journal of Optimization Theory and Applications}
  \bibinfo{volume}{174} (\bibinfo{year}{2017}) \bibinfo{pages}{83--107}.
\bibitem[{Habibli and Noori~Skandari(2019)}]{habibli2019fractional}
\bibinfo{author}{M.~Habibli}, \bibinfo{author}{M.~Noori~Skandari},
\newblock \bibinfo{title}{Fractional {C}hebyshev pseudospectral method for
  fractional optimal control problems},
\newblock \bibinfo{journal}{Optimal Control Applications and Methods}
  \bibinfo{volume}{40} (\bibinfo{year}{2019}) \bibinfo{pages}{558--572}.
\bibitem[{Ali et~al.(2019)Ali, Shamsi, Khosravian-Arab, Torres, and
  Bozorgnia}]{ali2019space}
\bibinfo{author}{M.~S. Ali}, \bibinfo{author}{M.~Shamsi},
  \bibinfo{author}{H.~Khosravian-Arab}, \bibinfo{author}{D.~F. Torres},
  \bibinfo{author}{F.~Bozorgnia},
\newblock \bibinfo{title}{A space--time pseudospectral discretization method
  for solving diffusion optimal control problems with two-sided fractional
  derivatives},
\newblock \bibinfo{journal}{Journal of Vibration and Control}
  \bibinfo{volume}{25} (\bibinfo{year}{2019}) \bibinfo{pages}{1080--1095}.
\bibitem[{Grady~Jr et~al.(2011)Grady~Jr, Daigger, Love, and
  Filipe}]{grady2011biological}
\bibinfo{author}{C.~L. Grady~Jr}, \bibinfo{author}{G.~T. Daigger},
  \bibinfo{author}{N.~G. Love}, \bibinfo{author}{C.~D. Filipe},
  \bibinfo{title}{Biological wastewater treatment}, \bibinfo{publisher}{CRC
  press}, \bibinfo{year}{2011}.
\bibitem[{Elgindy(2024)}]{elgindy2024optimal}
\bibinfo{author}{K.~T. Elgindy},
\newblock \bibinfo{title}{Optimal periodic control of unmanned aerial vehicles
  based on {F}ourier integral pseudospectral and edge-detection methods},
\newblock \bibinfo{journal}{Unmanned Systems}  (\bibinfo{year}{2024})
  \bibinfo{pages}{1--19}.
\bibitem[{Smeets(2010)}]{smeets2010stochastic}
\bibinfo{author}{P.~W. Smeets}, \bibinfo{title}{Stochastic modelling of
  drinking water treatment in quantitative microbial risk assessment},
  \bibinfo{publisher}{IWA Publishing}, \bibinfo{year}{2010}.
\bibitem[{Teel(2012)}]{teel2012hybrid}
\bibinfo{author}{A.~R. Teel}, \bibinfo{title}{Hybrid dynamical systems:
  {M}odeling, stability, and robustness}, \bibinfo{year}{2012}.
\bibitem[{Lowe et~al.(2022)Lowe, Qin, and Mao}]{lowe2022review}
\bibinfo{author}{M.~Lowe}, \bibinfo{author}{R.~Qin}, \bibinfo{author}{X.~Mao},
\newblock \bibinfo{title}{A review on machine learning, artificial
  intelligence, and smart technology in water treatment and monitoring},
\newblock \bibinfo{journal}{Water} \bibinfo{volume}{14} (\bibinfo{year}{2022})
  \bibinfo{pages}{1384}.
\bibitem[{Khalil and Grizzle(2002)}]{khalil2002nonlinear}
\bibinfo{author}{H.~K. Khalil}, \bibinfo{author}{J.~W. Grizzle},
  \bibinfo{title}{Nonlinear Systems}, volume~\bibinfo{volume}{3},
  \bibinfo{publisher}{Prentice hall Upper Saddle River, NJ},
  \bibinfo{year}{2002}.
\bibitem[{Li and Zhang(2011)}]{li2011survey}
\bibinfo{author}{C.~Li}, \bibinfo{author}{F.~Zhang},
\newblock \bibinfo{title}{A survey on the stability of fractional differential
  equations: {D}edicated to prof. {YS Chen} on the occasion of his 80th
  birthday},
\newblock \bibinfo{journal}{The European Physical Journal Special Topics}
  \bibinfo{volume}{193} (\bibinfo{year}{2011}) \bibinfo{pages}{27--47}.
\bibitem[{Guckenheimer and Holmes(2013)}]{guckenheimer2013nonlinear}
\bibinfo{author}{J.~Guckenheimer}, \bibinfo{author}{P.~Holmes},
  \bibinfo{title}{Nonlinear oscillations, dynamical systems, and bifurcations
  of vector fields}, volume~\bibinfo{volume}{42}, \bibinfo{publisher}{Springer
  Science \& Business Media}, \bibinfo{year}{2013}.

\end{thebibliography}

\end{document}